\documentclass[a4paper,10pt]{article}
\usepackage[utf8x]{inputenc}
\pdfoutput=1

\usepackage{times}
\usepackage{geometry}
\usepackage[T1]{fontenc}
\usepackage{color}
\usepackage{amsmath}
\usepackage{amssymb}
\usepackage{array}
\usepackage{amsthm}
\usepackage{graphicx}
\usepackage{hyperref}
\usepackage{setspace}
\usepackage{stmaryrd}

\theoremstyle{plain}
\newtheorem{thm}{Theorem}[section]
\newtheorem{prop}[thm]{Proposition}
\newtheorem{cor}[thm]{Corollary}
\newtheorem{lem}[thm]{Lemma}

\theoremstyle{definition}
\newtheorem{defn}[thm]{Definition}

\theoremstyle{remark}
\newtheorem{rem}[thm]{Remark}

\theoremstyle{plain}

\newcommand{\R}{\mathbb{R}}

\newcommand{\N}{\mathbb{N}}

\newcommand{\CP}{\mathbb{C}P}

\newcommand{\dimn}{\mathrm{dim}}
\newcommand{\grad}{\mathrm{grad}}

\newcommand{\spam}{\mathrm{span}}
\newcommand{\scal}{\mathrm{scal}}
\newcommand{\ric}{\mathrm{Ric}}
\newcommand{\trace}{\mathrm{tr}}
\newcommand{\kernel}{\mathrm{ker}}

\newcommand{\volume}{\mathrm{vol}}

\newcommand{\dv}{\text{ }dV}

\newcommand{\Diff}{\mathrm{Diff}}

\newcommand{\gradient}{\mathrm{grad}}

\newcommand*{\suchthat}[1]{\left|\vphantom{#1}\right.}

\renewcommand{\title}[1]{{\bfseries #1}\par}
\renewcommand{\author}[1]{\medskip{#1}\par\smallskip}
\newcommand{\affiliation}[1]{{\itshape #1}\par}
\newcommand{\email}[1]{E-mail:~\texttt{#1}\par}

\numberwithin{equation}{section}

\begin{document}
\begin{center}
\title{\LARGE Stability and instability of Ricci solitons}
\vspace{3mm}
\author{\Large Klaus Kröncke}
\vspace{3mm}
\affiliation{Universität Potsdam, Institut für Mathematik\\Am Neuen Palais 10\\14469 Potsdam, Germany}
\email{klaus.kroencke@uni-potsdam.de}
\end{center}
\vspace{2mm}
\begin{abstract}We consider the volume-normalized Ricci flow close to compact shrinking Ricci solitons. We show that if a compact Ricci soliton $(M,g)$ is a local maximum of Perelman's shrinker entropy,
any normalized Ricci flow starting close to it exists for all time and converges towards a Ricci soliton. If $g$ is not a local maximum of the shrinker entropy, we show that there exists a nontrivial normalized Ricci flow emerging from it.
These theorems are analogues of results in the Ricci-flat and in the Einstein case \cite{HM13,Kro13a}.
\end{abstract}
\section{Introduction}
 A Riemannian manifold $(M,g)$ is called a Ricci soliton if there exist a vector field $X\in\mathfrak{X}(M)$ and a constant $c\in\R$ such that
\begin{align*}
 \ric_g+L_Xg=c\cdot g.
\end{align*}
The soliton is called gradient, if $X=\gradient f$ for some $f\in C^{\infty}(M)$. We call $(M,g)$ expanding, if $c<0$, steady, if $c=0$ and shrinking, if $c>0$.
If $X=0$, we recover the definition of an Einstein metric with Einstein constant $c$. If $X\neq0$, we call the soliton nontrivial.

Ricci solitons were first introduced by Hamilton in the eighties \cite{Ham88}. They appear as self-similar solutions of the Ricci flow
\begin{align*}
 \dot{g}(t)=-2\ric_{g(t)}.
\end{align*}
More precisely, a Ricci flow starting at a Ricci soliton $g$ evolves by 
\begin{align*}
 g(t)=(1-2ct)\varphi^*_tg
\end{align*}
where $\varphi_t$ is a family of diffeomorphisms on $M$. In other words, Ricci solitons are fixed points of the Ricci flow considered as a dynamical system in the space of metrics modulo diffeomorphism and rescaling.
Ricci solitons satisfy many interesting geometric properties and have been studied extensively in recent research (see \cite{Cao10} for a survey).
They also appear in string theory \cite{Fri80,Kho08}.

 Compact Ricci solitons satisfy additional properties.
 In dimensions $n=2,3$, any compact Ricci soliton is of constant curvature \cite{Ham88,Ive93}. 
By Perelman's work \cite{Per02}, they are always gradient (see also \cite[Proposition 3.1]{ELM08}). Moreover, any expanding or steady compact Ricci soliton is necessarily Einstein \cite[Proposition 1.1]{Cao10}.
In dimensions $n\geq4$ there exist
examples of nontrivial compact Ricci solitons \cite{Koi90,Cao96,CG96,WZ04,PS10,DW11}. All known nontrivial compact Ricci solitons are K\"ahler.

In this work, we are going to study the Ricci flow as a dynamical system close to compact shrinking Ricci solitons.
Previously the behaviour of Ricci flow close to Ricci-flat metrics has been studied by various authors
\cite{GIK02,Ses06,Has12,HM13}. The Einstein case was studied in \cite{Ye93} and in a recent paper by the author \cite{Kro13a}.

In the context of Ricci solitons, it is more convenient to deal with the volume-normalized Ricci flow
\begin{align}\label{volumepreservingricciflow}\dot{g}(t)=-2\ric_{g(t)}+\frac{2}{n}\frac{1}{\volume(M,g(t))}\left(\int_M\scal_{g(t)}\dv_{g(t)}\right)g(t),
\end{align}
where $\scal_{g(t)}$, $\dv_{g(t)}$ and $\volume(M,g(t))$ are the scalar curvature, the volume element and the total volume with respect to $g(t)$, respectively.
\begin{defn}
 A compact Ricci soliton $(M,g)$ is called dynamically stable if for any $C^k$-neighbourhood $\mathcal{U}$ of $g$ in the space of metrics (where $k\geq3$), there exists a $C^{k+5}$-neighbourhood $\mathcal{V}\subset\mathcal{U}$
such that for any $g_0\in\mathcal{V}$, the normalized Ricci flow \eqref{volumepreservingricciflow} starting at $g_0$ exists for all $t\geq0$ and converges modulo diffeomorphism to an Einstein metric in $\mathcal{U}$ as $t\to\infty$.

We call a compact Ricci soliton $(M,g)$ dynamically unstable if there exists a nontrivial normalized Ricci flow defined on $(-\infty,0]$ which converges modulo diffeomorphism to $g$ as $t\to-\infty$.
\end{defn}
In his pioneering work \cite{Per02}, Perelman introduced the shrinker entropy $\nu$, which is a functional on the space of metrics and admits precisely the shrinking Ricci solitons as its critical points.
Perelman discovered the remarkable property that $\nu$ is nondecreasing under the Ricci flow and stays constant only at its critical points.
The shrinker entropy is the most important tool in proving our main theorems.
\begin{thm}[Dynamical stability]\label{Thm1}Let $(M,g)$ be a compact shrinking Ricci soliton. If $(M,g)$ is a local maximizer of $\nu$, it is dynamically stable.
\end{thm}
Note that we do not need to assume that $(M,g)$ is a strict local maximizer in the space of metrics modulo diffeomorphism and rescaling.
\begin{thm}[Dynamical instability]\label{Thm2}Let $(M,g)$ be a compact shrinking Ricci soliton. If $(M,g)$ is not a local maximizer of $\nu$, it is dynamically unstable.
\end{thm}
Observe that any compact shrinking Ricci soliton must be either dynamically stable or unstable, so we have a complete description of the Ricci flow as a dynamical system close to Ricci solitons.
Since non-shrinking compact Ricci solitons are Einstein, they are already covered by previous results.

In the Einstein case, we additionally were able to give geometric stability/instability conditions in terms of the conformal Yamabe invariant and the Laplace spectrum \cite{Kro13a}.
It would be intersting to find similar geometric conditions in the case of nontrivial Ricci solitons.
So far, we can characterize stability in terms of the eigenvalues of a second-order differential operator, supposed that an additional technical condition holds.
\begin{thm}
 Let $(M,g)$ be a compact shrinking Ricci soliton and suppose all infinitesimal solitonic deformations are integrable. Then $(M,g)$ is a local maximum of $\nu$ if and only if $\nu''\leq0$.
\end{thm}
The integrability condition means that all elements in the kernel of $\nu''$ can be integrated to curves of Ricci solitons. We show that this condition does not always hold.
\begin{thm}\label{nonintegrableISD}
 The complex projective space $(\CP^n,g_{fs})$ with the Fubini-Study metric admits nonintegrable infintesimal solitonic deformations.
\end{thm}

\section{Notation and conventions}
For the Riemann curvature tensor, we use the sign convention such that
 $R_{X,Y}Z=\nabla^2_{X,Y}Z-\nabla^2_{Y,X}Z$. Given a fixed metric, we equip the bundle of $(r,s)$-tensor fields (and any subbundle) with the natural pointwise scalar product induced by the metric.
By $S^pM$, we denote the bundle of symmetric $(0,p)$-tensors.
For a given $f\in C^{\infty}(M)$, we introduce some $f$-weighted differential operators. The $f$-weighted Laplacian (or Baker-Emery Laplacian) acting on $C^{\infty}(S^pM)$ is
\begin{align*}
 \Delta_f h=-\sum_{i=1}^n \nabla^2_{e_i,e_i}h+\nabla^2_{\gradient f}h.
\end{align*}
By the sign convention, $\Delta_f=(\nabla^*_f)\nabla$, where $\nabla^*_f$ is the adjoint of $\delta$ with respect to the weighted $L^2$-scalar product $\int_M\langle.,.\rangle e^{-f}\dv$.
The weighted divergence $\delta_f:C^{\infty}(S^pM)\to C^{\infty}(S^{p-1}M)$ and its formal adjoint $\delta_f^*\colon C^{\infty}(S^{p-1}M)\to C^{\infty}(S^pM)$ with respect to the weighted scalar product are given by
\begin{align*}\delta_f T(X_1,\ldots,X_{p-1})=&-\sum_{i=1}^n\nabla_{e_i}T(e_i,X_1,\ldots,X_{p-1})+T(\gradient f,X_1,\ldots,X_{p-1}),\\
            \delta_f^*T(X_1,\ldots,X_p)=&\frac{1}{p}\sum_{i=0}^{p-1}\nabla_{X_{1+i}}T(X_{2+i},\ldots,X_{p+i}),
\end{align*}
where the sums $1+i,\ldots,p+i$ are taken modulo $p$. If $f$ is constant, we recover the usual notions of Laplacian and divergence. In this case, we will drop the $f$ in the notation.
For $\omega\in\Omega^1(M)$, we have $\delta_f^*\omega=L_{\omega^{\sharp}}g$ where $\omega^{\sharp}$ is the sharp of $\omega$.
Thus, $\delta_f^*(\Omega^1(M))$ is the tangent space of the manifold $g\cdot \Diff(M)=\left\{\varphi^*g|\varphi\in\Diff(M)\right\}$.
Throughout, any manifold $M$ will be compact and any metric considered on $M$ will be smooth, unless the contrary is explicitly asserted.
   \section{The shrinker entropy}
Let $g$ be a Riemannian metric, $f\in C^{\infty}(M)$, $\tau>0$ and define
\begin{align*}\mathcal{W}(g,f,\tau)=\frac{1}{(4\pi\tau)^{n/2}}\int_M [\tau(|\nabla f|^2_g+\scal_g)+f-n]e^{-f}\dv.
\end{align*}
Let\index{$\mu_(g,\tau)$}
 \begin{align*}\mu(g,\tau)&=\inf \left\{\mathcal{W}(g,f,\tau)\suchthat{f\in C^{\infty}(M),\frac{1}{(4\pi\tau)^{n/2}}\int_M e^{-f}\dv_g=1}f\in C^{\infty}(M),\frac{1}{(4\pi\tau)^{n/2}}\int_M e^{-f}\dv_g=1\right\}.
 \end{align*}
For any fixed $\tau>0$, the infimum is finite and is realized by a smooth function \cite[Lemma 6.23 and 6.24]{CC07}.
We define the shrinker entropy\index{shrinker entropy} as
\begin{align*}\nu(g)&=\inf \left\{\mu(g,\tau)\mid\tau>0\right\}.
\end{align*}
Recall also the definition of Perelman's $\lambda$-functional
\begin{equation}\label{perelmanslambda}\lambda(g)=\inf_{\substack{f\in C^{\infty}(M)\\\int_M e^{-f}\dv_g=1}}\int_M(\scal_g+|\nabla f|^2_g)e^{-f}\dv_g.\end{equation}
If $\lambda(g)>0$,
 then $\nu(g)$ is finite and realized by some $\tau_g>0$ (see \cite[Corollary 6.34]{CC07}).
In this case, a pair $(f_g,\tau_g)$ realizing $\nu(g)$ satisfies the equations\index{$f_g$, minimizer realizing $\nu(g)$}\index{$\tau_g$, minimizer realizing $\nu(g)$}
\begin{align}\label{nueulerlagrange}\tau(2\Delta f+|\nabla f|^2-\scal)-f+n+\nu&=0,\\
\label{nueulerlagrange1.5}\frac{1}{(4\pi\tau)^{n/2}}\int_M f e^{-f}\dv&=\frac{n}{2}+\nu,
\end{align}
see e.g. \cite[p.\ 5]{CZ12}. On the other hand, $\nu(g)$ is not finite if $\lambda(g)<0$ \cite[p.\ 244]{CC07}.
\begin{rem}Note that $\mathcal{W}$ posseses the symmetries $\mathcal{W}(\varphi^*g,\varphi^*f,\tau)=\mathcal{W}(g,f,\tau)$ for $\varphi\in\Diff(M)$
and $\mathcal{W}(\alpha g,f,\alpha \tau)=\mathcal{W}(g,f,\tau)$ for $\alpha>0$. Therefore,
$\nu(g)=\nu(\alpha \cdot\varphi^{*}g)$ for any $\varphi\in\Diff(M)$ and $\alpha>0$.
\end{rem}
\begin{rem}\label{lsc}The shrinker entropy is upper semicontinuous with respect to the $C^2$-topology if defined. Let $g$ be fixed and $(f_g,\tau_g)$ be a minimizing pair.
 Let $g_i\to g$ in $C^2$ and $v_{i}\in C^{\infty}(M)$ such that $e^{-v_i}\dv_{g_i}=e^{-f_{g}}\dv_g$. Then,
\begin{align*}
 \nu(g)=\mathcal{W}(g,f_g,\tau_g)=\lim_{i\to\infty}\mathcal{W}(g_i,v_i,\tau_g)\geq\limsup_{i\to\infty}\nu(g_i).
\end{align*}

\end{rem}

\begin{prop}[First variation of $\nu$]\label{firstnu}\index{first variation!of $\nu$}Let $(M,g)$ be a Riemannian manifold. Then the first variation of $\nu$ is given by
\begin{align*}\nu(g)'(h)=-\frac{1}{(4\pi\tau_g)^{n/2}}\int_M\left\langle \tau_g(\ric+\nabla^2 f_g)-\frac{1}{2}g,h\right\rangle e^{-f_g}\dv_g,
\end{align*}
where $(f_g,\tau_g)$ realizes $\nu(g)$.
\end{prop}
 \begin{proof}See e.g.\ \cite[Lemma 2.2]{CZ12}. 
 \end{proof}
\noindent
By scale and diffeomorphism invariance, it is easy to see that $\nu$ is nondecreasing under the Ricci flow and that $\nu$ stays constant if and only if
\begin{align}\label{Riccisoliton}
 \ric+\nabla^2 f_g=\frac{1}{2\tau_g}g,
\end{align}
i.e.\ if $g$ is a critical point of $\nu$. Observe that in this case, the pair $(f_g,\tau_g)$ realizing $\nu(g)$ is unique.
\begin{rem}
 A metric $g$ is a critical point of $\nu$ if and only if $g$ is a gradient shrinking Ricci soliton. Suppose that \eqref{Riccisoliton} holds, then $g$ is a gradient shrinking Ricci soliton since
$\nabla^2 f_g=\frac{1}{2}L_{\gradient f_g}g$ and $\frac{1}{2\tau_g}>0$. 
Conversely, any gradient shrinking Ricci soliton $g$ has positive scalar curvature \cite[Proposition 1]{Ive93}.
Therefore, $\nu(g)$ is finite, since $\lambda(g)\geq\min\scal_g>0$. Since $\nu$ stays constant along the flow starting at $g$, we necessarily have \eqref{Riccisoliton}.
\end{rem}
\begin{prop}[Second variation of $\nu$]\label{secondnu}Let $(M,g)$ be a gradient shrinking Ricci soliton. Then the second variation of $\nu$ at $g$ is given by
 \begin{align}
  \nu''_{g}(h)=\frac{\tau}{(4\pi\tau)^{n/2}}\int_M \langle N h,h\rangle e^{-f}\dv
 \end{align}
where $(f,\tau)$ is the minimizing pair realizing $\nu$. The stability operator $N$ is given by
\begin{align}
 Nh=-\frac{1}{2}\Delta_f h+\mathring{R}h+\delta^*_f(\delta_f(h))+\frac{1}{2}\nabla^2 v_h-\ric \frac{\int_M \langle \ric,h\rangle e^{-f}\dv}{\int_M \scal e^{-f}\dv}.
\end{align}
Here, $\mathring{R}h(X,Y)=\sum_{i=1}^nh(R_{e_i,X}Y,e_i)$ and $v_h$ is the unique solution of
\begin{align}\label{v_h}
 (-\Delta_f+\frac{1}{2\tau})v_h=\delta_f(\delta_f(h)).
\end{align}
\end{prop}
\begin{proof}
 See \cite[Theorem 1.1]{CZ12}.
\end{proof}
\begin{rem}\label{decompremark}In the following, we explain how this variational formula can be substantially simplified.
 The operator $N$ is formally self-adjoint with respect to the weighted measure because
\begin{align*}
 \int_M \langle \nabla^2 v_h,k\rangle e^{-f}\dv=\int_M \langle v_h,\delta_f(\delta_f(k))\rangle e^{-f}\dv=\int_M \langle v_h,(-\Delta_f+\frac{1}{2\tau})v_k\rangle e^{-f}\dv
\end{align*}
and the right hand side is symmetric in $h$ and $k$. The other summands of $N$ are clearly self-adjoint. Therefore,
\begin{align*}
 \frac{d^2}{dsdt}\bigg\vert_{s,t=0}\nu(g+th+sk)=\frac{\tau}{(4\pi\tau)^{n/2}}\int_M \langle N h,k\rangle e^{-f}\dv.
\end{align*}
By scale and diffeomorphism invariance of $\nu$, we have, for an arbitrary $k\in C^{\infty}(S^2M)$,
\begin{align*}
 \frac{\tau}{(4\pi\tau)^{n/2}}\int_M \langle N h,k\rangle e^{-f}\dv=0,\qquad\text{ if }h\in \R\cdot g\oplus \delta_f^*(\Omega^1(M))=\R\cdot \ric\oplus \delta_f^*(\Omega^1(M)).
\end{align*}
The equality of the direct sums holds because $g$ is a Ricci soliton.
Thus $\nu''$ is only nontrivial on the orthogonal component of $\ric\cdot g\oplus \delta_f^*(\Omega^1(M))$, given by
\begin{align*}
 V=\left\{h\in \Gamma(S^2M)\text{ }\bigg\vert \text{ }\delta_f(h)=0 \text{ and }\int_M \langle\ric,h\rangle e^{-f}\dv=0\right\},
\end{align*}
and the stability operator $N|_{V}:V\to V$ is of the form
\begin{align*}
 N=-\frac{1}{2}\Delta_f+\mathring{R}.
\end{align*}
We also have $\delta_f\ric=0$ \cite[Lemma 3.1]{CZ12} and thus, $\delta_f^{-1}(0)=\R\cdot\ric\oplus V$.
\end{rem}
\begin{defn}
 Let $(M,g)$ be a Ricci soliton. We call the soliton linearly stable if all eigenvalues of $N$ are nonpositive and linearly unstable otherwise. If $N$ is nonpositive and $\kernel N\cap V\neq \left\{0\right\}$, we call the soliton
neutrally linearly stable.
\end{defn}
\begin{rem}\label{linstabilityremark}The round sphere is linearly stable; the complex projective space with the Fubini-Study metric is neutrally linearly stable. 
Any product of positive Einstein manifolds is unstable. Some more examples are discussed in \cite{CHI04,CH13,HHM14}. 
It is conjectured that all compact $4$-dimensional nontrivial Ricci solitons are linearly unstable \cite[p.\ 29]{Cao10}. Due to a result by Hall and Murphy \cite{HM11},
any k\"ahlerian shrinking Ricci soliton is linearly unstable if $\dimn H^{1,1}(M)>1$. This results applies to all known nontrivial Ricci solitons in any dimension.
\end{rem}

\section{Analyticity and a Lojasiewicz-Simon inequality}
A necessary tool in proving our stability and instability theorems is a Lojasiewicz-Simon inequality for $\nu$ which will be the main theorem of this section.
To prove this, $\nu$ needs to be analytic. This is the content of the following
\begin{lem}\label{analyticnu}Let $(M,g_0)$ be a gradient shrinking Ricci soliton. Then there exists a $C^{2,\alpha}$-neighbourhood $\mathcal{U}$ of $g_0$ in the space of metrics such 
that the minimizing pair $(f_g,\tau_g)$ realizing $\nu(g)$ is unique and depends analytically\index{analytic} on the metric. Moreover, the map $g\mapsto \nu(g)$ is analytic on $\mathcal{U}$.
\end{lem}
\begin{rem}Previously, analytic dependence of $\mu(g,1/2)$ on $g$ in the neighbourhood of a Ricci soliton was shown in \cite[Lemma 2.2]{SW13}.
 The proof presented here turns out to be similar but some subtleties occur due to the presence of the scale parameter $\tau$.
\end{rem}

\begin{proof}[Proof of Lemma \ref{analyticnu}]We use the implicit function theorem for Banach manifolds. We define a map $H(g,f,\tau)=\tau(2\Delta f+|\nabla f|^2-\scal)-f+n$.
Let
$$C^{k,\alpha}_{g_0}(M)=\left\{u\in C^{k,\alpha}(M)\hspace{1mm}\bigg\vert\hspace{1mm}\int_M u e^{-f_{g_0}}\dv_{g_0}=0\right\}$$
and let $\mathcal{M}^{2,\alpha}$ be the Banach manifold of $C^{2,\alpha}$-metrics. 
Define
 \begin{align*}L\colon\mathcal{M}^{2,\alpha}\times C^{2,\alpha}(M)\times\R_+&\to C^{0,\alpha}_{g_0}(M)\times\R\times\R,\\
                (g,f,\tau)&\mapsto(L_1,L_2,L_3),
\end{align*}
where the three components are given by
\begin{align*}L_1(g,f,\tau)&=H(g,f,\tau)-\frac{1}{(4\pi\tau_{g_0})^{n/2}}\int_MH(g,f,\tau)e^{-f_{g_0}}\dv_{g_0},\\
              L_2(g,f,\tau)&=\frac{1}{(4\pi\tau)^{n/2}}\int_M fe^{-f}\dv_g-\frac{n}{2}+\frac{1}{(4\pi\tau)^{n/2}}\int_M H(g,f,\tau)e^{-f_g}\dv_g,\\
              L_3(g,f,\tau)&=\frac{1}{(4\pi\tau)^{n/2}}\int_M e^{-f}\dv_g-1.
\end{align*}
This is an analytic map between Banach manifolds.
We have $L(g,f,\tau)=(0,0,0)$ if and only if there exists a constant $c\in\R$ such that the set of equations 
\begin{align}\tau(2\Delta f+|\nabla f|^2-\scal)-f+n&=c,\\
              \frac{1}{(4\pi\tau)^{n/2}}\int_M fe^{-f}\dv-\frac{n}{2}&=-c,\\
                 \frac{1}{(4\pi\tau)^{n/2}}\int_M e^{-f}\dv&=1
\end{align}
is satisfied. Now we compute the differential of $L$ at $(g_0,f_{g_0},\tau_{g_0})$ restricted to $R=C^{2,\alpha}(M)\times\R$.
Let $F_{g_0}=f_{g_0}-\frac{n}{2}-\nu(g_0)$ and
\begin{align*}V=\left\{u\in C^{2,\alpha}_{g_0}(M)\hspace{1mm}
\bigg\vert\hspace{1mm} \int_M u\cdot F_{g_0}e^{-f_{g_0}}\dv_{g_0}=0\right\},\\
W=\left\{u\in C^{0,\alpha}_{g_0}(M)\hspace{1mm}
\bigg\vert\hspace{1mm} \int_M u\cdot F_{g_0} e^{-f_{g_0}}\dv_{g_0}=0\right\}.
\end{align*}
By \eqref{nueulerlagrange1.5} and smoothness of $f_{g_0}$, we have $F_{g_0}\in C^{k,\alpha}_{g_0}(M)$ for all $k\geq0$. Moreover, we have the decompositions
\begin{align*}
  C^{2,\alpha}(M)&\cong V\oplus \spam(F_{g_0})\oplus \R,\\
  C^{0,\alpha}_{g_0}(M)&\cong W\oplus \spam(F_{g_0}),
 \end{align*}
where the last factor in the first decomposition represents the constant functions. Consider the differential of $L$ restricted to $R$ as a linear map
\begin{align*}dL_{(g_0,f_{g_0},\tau_{g_0})}\big\vert_{R}: V\oplus\spam(F_{g_0})\oplus \R\oplus\R\to W\oplus \spam(F_{g_0})\oplus \R\oplus\R.
 \end{align*}
Straightforward calculations, using the Euler-Lagrange equations \eqref{nueulerlagrange} and \eqref{nueulerlagrange1.5}, show that it is equal to
\begin{align*}
 dL_{(g_0,f_{g_0},\tau_{g_0})}\big\vert_{R}=\begin{pmatrix}2\tau_{g_0}\Delta_{f_{g_0}}-1 & 0 & 0 & 0\\
0 & 1 & 0 & \frac{1}{\tau_{g_0}}\\
 0 & -\left\| F_{g_0}\right\|^2_{L^2(v\dv)} &-\frac{n}{2}-\nu(g_0) & -\frac{n}{2\tau_{g_0}}(\frac{n}{2}+\nu(g_0)+1)\\
0 & 0 & -1 & -\frac{n}{2\tau_{g_0}}
  \end{pmatrix},
\end{align*}
where $L^2(v\dv)$ is the $L^2$ norm with respect to the $v$-weighted measure, with $v=\frac{1}{(4\pi\tau_{g_0})^{n/2}}e^{-f_{g_0}}$.
In the following, we show that this map is an isomorphism. From \eqref{nueulerlagrange}, we conclude that $\Delta_{f_{g_0}}F_{g_0}=\frac{1}{\tau_{g_0}} F_{g_0}$. Thus, the map
\begin{align*}
 2\tau_{g_0}\Delta_{f_{g_0}}-1:V\to W
\end{align*}
is well defined. By \cite[Proposition 3.1]{FS13}, the smallest nonzero eigenvalue of $\Delta_{f_{g_0}}$ satisfies $\lambda>\frac{1}{2\tau_{g_0}}$.
This implies invertibility of the above map.
 It remains to consider the lower right $3\times3$-block which we denote by $A$. We have
\begin{align*}
 \det(A)=\frac{1}{\tau_{g_0}}\left(\left\|F_{g_0}\right\|_{L^2(v\dv)}^2-\frac{n}{2}\right).
\end{align*}
Since $F_{g_0}$ is an eigenfunction of the weighted Laplacian to the eigenvalue $\frac{1}{\tau_{g_0}}$,
\begin{align*}
 \left\|F_{g_0}\right\|_{L^2(v\dv)}^2=\tau_{g_0}(\Delta_f F_{g_0},F_{g_0})_{L^2(v\dv)}=\tau_{g_0}\left\|\nabla F_{g_0}\right\|_{L^2(v\dv)}^2=\tau_{g_0}\left\|\nabla f_{g_0}\right\|_{L^2(v\dv)}^2.
\end{align*}
By definition of $\nu(g_0)$,
\begin{align*}
 \tau_{g_0}\left\|\nabla f_{g_0}\right\|_{L^2(v\dv)}^2&=\frac{1}{(4\pi\tau_{g_0})^{n/2}}\int_M \tau_{g_0}|\nabla  f_{g_0}|^2e^{-f_{g_0}}\dv\\
      & =\nu(g_0)-\frac{1}{(4\pi\tau_{g_0})^{n/2}}\int_M [\tau_{g_0}\scal+f_{g_0}-n]e^{-f_{g_0}}\dv\\
      &=\frac{n}{2}-\frac{1}{(4\pi\tau_{g_0})^{n/2}}\int_M \tau_{g_0}\scal e^{-f_{g_0}}\dv.
\end{align*}
Therefore,
\begin{align*}
  \det(A)=-\frac{1}{(4\pi\tau_{g_0})^{n/2}}\int_M \scal \text{ }e^{-f_{g_0}}\dv<0
\end{align*}
because the scalar curvature of $g_0$ is positive \cite[Proposition 1]{Ive93}. In summary, we have shown that $dL_{(g_0,f_{g_0},\tau_{g_0})}|_{R}$ is invertible.
                             By the implicit function theorem\index{implicit function theorem} for Banach manifolds\index{Banach manifold}, there exists a neighbourhood $\mathcal{U}\subset \mathcal{M}^{{2,\alpha}}$ of $g_0$ and an analytic\index{analytic} map
$P\colon\mathcal{U}\to C^{2,\alpha}(M)\times\R_+$ such that $L(g,P(g))=0$. Moreover, there exists a neighbourhood $\mathcal{V}\subset C^{2,\alpha}(M)\times\R_+ $ such that 
for any $(g,f,\tau)\in\mathcal{U}\times\mathcal{V}$, we have
$L(g,f,\tau)=0$ if and only if $(f,\tau)=P(g)$.

Now we claim that on a smaller neighbourhood $\mathcal{U}_1\subset\mathcal{U}$, there is a unique pair of minimizers\index{minimizer} in the definition of $\nu$ and it is equal to $P(g)$. Suppose this is not the case.
Then there exist a sequence\index{sequence} $g_i$ of metrics such that $g_i\to g_0$ in $C^{2,\alpha}$ and pairs of minimizers $(f_{g_i},\tau_{g_i})$ such that $P(g_i)\neq (f_i,\tau_{g_i})$ for all $i\in\N$.
 By substituting $w_{g_i}^2=e^{-f_{g_i}}$, we see that the pair $(w_{g_i},\tau_{g_i})$ is a minimizer of the functional
\begin{align*}\tilde{\mathcal{W}}(g_i,w,\tau)=\frac{1}{(4\pi\tau)^{n/2}}\int_M [\tau(4|\nabla w|^2+\scal_gw^2)-\log (w^2)w^2-nw^2]\dv_{g_i}
\end{align*}
under the constraint $\frac{1}{(4\pi\tau)^{n/2}}\int_M w^2 \dv_{g_i}=1$. It satisfies the pair of equations
\begin{align}\label{nueulerlagrange0}-\tau_{g_i}(4\Delta w_{g_i}+\scal_{g_i}w_{g_i})-2\log(w_{g_i})w_{g_i}+nw_{g_i}+\nu(g_i)w_{g_i}&=0,\\
-\frac{1}{(4\pi\tau_{g_i})^{n/2}}\int_M w_{g_i}^2\log w_{g_i}^2 \dv_{g_i}&=\frac{n}{2}+\nu(g_i).
\end{align} 
By upper semicontinuity, $\nu(g_i)\leq C_1$.

Now we show that there exist constants $C_2,C_3>0$ such that $C_2\leq \tau_{g_i}\leq C_3$. Suppose this is not the case. By \cite[Lemma 6.30]{CC07},
we have a lower estimate
\begin{align*}\nu(g_i)=\mu(g_i,\tau_{g_i})&\geq (\tau_{g_i}-1)\lambda(g_i)-\frac{n}{2}\log\tau_{g_i}-C_4(g_i)
                                              \geq (\tau_{g_i}-1)C_5-\frac{n}{2}\log\tau_{g_i}-C_6.
\end{align*}
We have $C_5>0$ because $\lambda(g)$ is uniformly bounded on $\mathcal{U}$ and positive.
 The constant $C_4(g)$ depends on the Sobolev constant\index{Sobolev!constant} and the volume. Now if $\tau_{g_i}$ converges to $0$ or $\infty$, $\nu(g_i)$ diverges, which causes the contradiction. 
Observe that we also obtained a lower bound on $\nu(g_i)$.

Next, we show that $\left\|\nabla w_{g_i}\right\|_{L^2}$ is bounded. Choose $\epsilon>0$ so small that $2+2\epsilon\leq\frac{2n}{n-2}$.
By Jensen's inequality\index{Jensen's inequality} and the bounds on $\tau_{g_i}$,
\begin{align*}\int_M w_{g_i}^2\log w_{g_i}^2 \dv_{g_i}=&\frac{1}{\epsilon}\int_M w_{g_i}^2 \log w_{g_i}^{2\epsilon}\dv_{g_i}\\
                                                        \leq& \frac{1}{\epsilon}\left\| w_{g_i}\right\|^2_{L^2}\log\left(\frac{1}{\left\| w_{g_i}\right\|^2_{L^2}} \int_M w_{g_i}^{2+2\epsilon}\dv_{g_i}\right)\\
                                                                 =&\frac{1}{\epsilon}(4\pi\tau_{g_i})^{n/2}\log\left((4\pi\tau_{g_i})^{-n/2}\ \int_M w_{g_i}^{2+2\epsilon}\dv_{g_i}\right)\\
                                                                 \leq& C_7\log\left(\int_M w_{g_i}^{2+2\epsilon}\dv_{g_i}\right)+C_8.
\end{align*}
By the Sobolev inequality\index{Sobolev!inequality},
\begin{align*}\int_M w_{g_i}^{2+2\epsilon}\dv_{g_i}\leq& C_9(\left\|\nabla w_{g_i}\right\|^2_{L^2}+\left\|w_{g_i}\right\|^2_{L^2})^{1+\epsilon}
                                                     \leq C_9 (\left\|\nabla w_{g_i}\right\|^2_{L^2}+C_{10})^{1+\epsilon}.
\end{align*}
In summary, we have
\begin{align*}C_1&\geq \frac{1}{(4\pi\tau)^{n/2}}\int_M [\tau(4|\nabla w_{g_i}|^2+\scal_{g_i}w_{g_i}^2)-\log (w_{g_i}^2)w_{g_i}^2-nw^2]\dv_{g_i}\\
                 &\geq C_{11}\left\|\nabla w_{g_i}\right\|^2_{L^2}-C_{12}\log(\left\|\nabla w_{g_i}\right\|^2_{L^2}+C_{10})-C_{13},
\end{align*}
which shows that $\left\|\nabla w_{g_i}\right\|_{L^2}$ is bounded.

Now we continue with a bootstrap\index{bootstrap} argument. By Sobolev embedding, the bound on $\left\| w_{g_i}\right\|_{H^1}$ implies a bound on $\left\| w_{g_i}\right\|_{L^{2n/(n-2)}}$.
Let $p=2n/(n-2)$ and choose some $q$ slightly smaller than $p$. By elliptic regularity\index{elliptic regularity} and (\ref{nueulerlagrange0}),
\begin{align*}\left\|w_{g_i}\right\|_{W^{2,q}}\leq C_{14}(\left\|w_{g_i}\log w_{g_i}\right\|_{L^q}+\left\|w_{g_i}\right\|_{L^q}).
\end{align*}
Since for any $\beta>1$, $|x\log x|\leq |x|^{\beta}$ for $|x|$ large enough, we have
\begin{align*}\left\|w_{g_i}\log w_{g_i}\right\|_{L^q}\leq C_{15}(\volume(M,g_i))+\left\|w_{g_i}\right\|_{L^p}\leq C_{16}+\left\|w_{g_i}\right\|_{L^p}.
\end{align*}
Thus, $\left\|w_{g_i}\right\|_{W^{2,q}}\leq C(q)$. Using Sobolev embedding\index{Sobolev!embedding}, we obtain bounds on $\left\|w_{g_i}\right\|_{L^{p'}}$ for some $p'>p$. 
From (\ref{nueulerlagrange0}) again, we have bounds on $\left\|w_{g_i}\right\|_{W^{2,q'}}$ for any $q'<p'$.
Using these arguments repetitively, we obtain
$\left\|w_{g_i}\right\|_{W^{2,q}}\leq C(q)$ for all $q\in (1,\infty)$. Again by elliptic regularity\index{elliptic regularity},
\begin{align*}\left\|w_{g_i}\right\|_{C^{2,\alpha}}&\leq C_{17}(\left\|w_{g_i}\log w_{g_i}\right\|_{C^{0,\alpha}}+\left\| w_{g_i}\right\|_{C^{0,\alpha}})
                                                   \leq C_{18}((\left\|w_{g_i}\right\|_{C^{0,\alpha}})^{\gamma}+\left\| w_{g_i}\right\|_{C^{0,\alpha}})
\end{align*}
for some $\gamma>1$. For some sufficiently large $q$, we have, by Sobolev embedding\index{Sobolev!embedding},
\begin{align*}\left\|w_{g_i}\right\|_{C^{0,\alpha}}\leq C_{19}\left\|w_{g_i}\right\|_{W^{1,q}}\leq C_{19}\cdot C(q).
\end{align*}
We finally obtained an upper bound on $\left\|w_{g_i}\right\|_{C^{2,\alpha}}$. Thus there exists a subsequence, again denoted by $(w_{g_i},\tau_{g_i})$, which converges in $C^{2,\alpha'}$, $\alpha'<\alpha$, to some limit $(w_{\infty},\tau_{\infty})$. By Remark
\ref{lsc},
\begin{align*}\nu(g_0)\geq\lim_{i\to\infty} \nu(g_i)=\lim_{i\to\infty} \tilde{\mathcal{W}}(g_i,w_{g_i},\tau_{g_i})=\tilde{\mathcal{W}}(g_0,w_{\infty},\tau_{\infty})\geq \nu(g_0),
\end{align*}
and therefore, $(w_{\infty},\tau_{\infty})=(w_{g_0},\tau_{g_0})$ because the minimizing pair\index{minimizer} is unique at $g_0$. Moreover, by resubstituting, 
\begin{align*}(f_{g_i},\tau_{g_i})\to (f_{\infty},\tau_{\infty})=(f_{g_0},\tau_{g_0})
\end{align*} in $C^{2,\alpha'}$.
Because the pair $(f_{g_i},\tau_{g_i})$ satisfies \eqref{nueulerlagrange} and \eqref{nueulerlagrange1.5}, $L(g_i,f_{g_i},\tau_{g_i})=0$ and the implicit function argument from above implies that $P(g_i)=(f_{g_i},\tau_{g_i})$ for large $i$. This proves the claim by contradiction.
We also have shown that the constant $c$ appearing above is equal to $-\nu(g)$.
 Since the map $g\mapsto (f_g,\tau_g)$ is analytic\index{analytic}, the map
\begin{align*}g\mapsto \nu(g)=-\tau_g(2\Delta f_g+|\nabla f_g|^2-\scal_g)+f_g-n
\end{align*}
is also analytic. This proves the lemma.
\end{proof}
\begin{lem}\label{f',tau}Let $(M,g_0)$ be a gradient shrinking Ricci soliton. Then there exists a $C^{2,\alpha}$-neighbourhood
$\mathcal{U}$ of $g_0$ in the space of metrics and a constant $C>0$ such that
\begin{align*}
\left\|\frac{d}{dt}\bigg\vert_{t=0}f_{g+th}\right\|_{C^{2,\alpha}}&\leq C  \left\|h\right\|_{C^{2,\alpha}},\quad\left\|\frac{d}{dt}\bigg\vert_{t=0}f_{g+th}\right\|_{H^i}\leq C  \left\|h\right\|_{H^i},\quad i=0,1,2,\\
\left|\frac{d}{dt}\bigg\vert_{t=0}\tau_{g+th}\right|&\leq C  \left\|h\right\|_{L^2}
\end{align*}
for all $g\in\mathcal{U}$.
\end{lem}
\begin{proof}Let the maps $P,L$ and the space $R$ be as in the proof of the previous lemma.
Note that if $\mathcal{U}$ is small enough, the map
\begin{align*}h\mapsto \left(\frac{d}{dt}\bigg\vert_{t=0}f_{g+th},\frac{d}{dt}\bigg\vert_{t=0}\tau_{g+th}\right)
\end{align*}
is precisely the differential of $P$. By the chain rule,
\begin{align}\label{dP}dP_{g}=-\left(dL_{g,f_{g},\tau_{g}}|_R\right)^{-1}\circ dL_{g,f_{g},\tau_{g}}|_{C^{2,\alpha}(S^2M)}.
 \end{align}
First, we compute $dL_{g,f,\tau}|_{C^{2,\alpha}(S^2M)}$. We have
\begin{align*}
dH(g,h,f,\tau):=\frac{d}{dt}\bigg\vert_{t=0}H(g+th,f,\tau)=2\tau\Delta'f-\tau h(\grad f,\grad f)-\scal'.
\end{align*}
Then,
\begin{align*}dL_1(h)&=dH(g,h,f,\tau)-\frac{1}{(4\pi\tau_{g_0})^{n/2}}\int_MdH(g,h,f,\tau)e^{-f_{g_0}}dV_{g_0},\\
              dL_2(h)&=\frac{1}{(4\pi\tau)^{n/2}}\frac{1}{2}\int_M fe^{-f}\trace h\dv_g+\frac{1}{(4\pi\tau_{g_0})^{n/2}}\int_M dH(g,h,f,\tau)e^{-f_g}\dv_{g_0},\\
              dL_3(h)&=\frac{1}{(4\pi\tau)^{n/2}}\frac{1}{2}\int_M e^{-f}\trace h\dv_g
\end{align*}
by the first variation of the volume element.
The first variation of the Laplacian and the scalar curvature are
\begin{align*}
\frac{d}{dt}\bigg\vert_{t=0}\Delta_{g+th} f=&\langle h,\nabla^2 f\rangle-\left\langle\delta h+\frac{1}{2}\nabla\trace h,\nabla f\right\rangle,\\
 \frac{d}{dt}\bigg\vert_{t=0}\scal_{g+th}=&\Delta_g(\trace_g h)+\delta_g(\delta_g h)-\langle \ric_g,h\rangle_g,
\end{align*}
c.f.\ \cite[pp.\ 62-64]{Bes08}.
Therefore, we have the estimates
\begin{align*}
 \left\|dL_{g,f,\tau}|_{C^{2,\alpha}(S^2M)}(h)\right\|_{C^{0,\alpha}}&\leq C_1\left\| h\right\|_{C^{2,\alpha}},\\
 \left\|dL_{g,f,\tau}|_{C^{2,\alpha}(S^2M)}(h)\right\|_{H^{i-2}}&\leq C_2\left\| h\right\|_{H^i},
\end{align*}
where $i=0,1,2$. Now we consider $dL_{g,f,\tau}|_R$. 
This is essentially an elliptic operator which is invertible at $(g_0,f_{g_0},\tau_{g_0})$. By continuity, it is also invertible on a small $C^{2,\alpha}$-neighbourhood of the tuple.
By elliptic estimates, we conclude from \eqref{dP} that
\begin{align*}
 \left\|dP_{g}(h)\right\|_{C^{2,\alpha}}&\leq C_3\left\| h\right\|_{C^{2,\alpha}},\qquad \left\|dP_{g}(h)\right\|_{H^{i}}\leq C_4\left\| h\right\|_{H^i}
\end{align*}
for $g$ close to $g_0$. This finishes the proof of the lemma.

\end{proof}
To prove the main result of this section, we need a form of the slice theorem. Recall that Ebin's slice theorem \cite{Eb70} states the following: For any Riemannian metric $g$ on a compact manifold, there exists neighbourhood $\mathcal{U}$ of $g$
and a submanifold $\mathcal{S}_g\subset\mathcal{M}$ tangent to $\delta^{-1}_g(0)$ such that any metric in $\mathcal{U}$ is isometric to a unique metric in $\mathcal{S}_g$. We call $\mathcal{S}_g$ a slice of the action of $\Diff(M)$. 

For our purpose, it is more convenient to use a slice tangent to $\delta^{-1}_{g,f}(0)$ where $f$ is the potential function of the Ricci soliton. In fact, we can choose the slice to be affine, i.e.\
\begin{align}\label{affineslice}
 \mathcal{S}_{g,f}=\left\{g+h\mid\delta_{g,f}(h)=0\right\}.
\end{align}
By \cite[Lemma A.5]{Ach12}, there exists a $C^{1+k,\alpha}$-neighbourhood $\mathcal{U}$ of $g$ such that any metric in $\mathcal{U}$ is isometric to a unique metric in $\mathcal{S}_{g,f}\cap\mathcal{U}$.
\begin{thm}[Lojasiewicz-Simon inequality]\label{nonintegrablegradientII}\index{Lojasiewicz-Simon inequality}Let $(M,g_0)$ be a gradient shrinking Ricci soliton. Then there exists a $C^{2,\alpha}$ neighbourhood $\mathcal{U}$
of $g_0$ and constants $\sigma\in[1/2,1)$, $C>0$ such that
\begin{align}\label{gradientnu2}|\nu(g)-\nu(g_0)|^{\sigma}\leq C \left\|\tau(\ric_g+\nabla^2 f_g)-\frac{1}{2}g\right\|_{L^2}
\end{align}
for all $g\in\mathcal{U}$.
\end{thm}
\begin{proof}Since both sides are diffeomorphism invariant\index{diffeomorphism!invariance}, it suffices to show the inequality on a slice\index{slice} to the action of the diffeomorphism group\index{diffeomorphism!group}. Let
\begin{align*}\mathcal{S}_{g_0,f_{g_0}}=\mathcal{U}\cap\left\{g_0+h \text{ }\bigg\vert\text{ }\delta_{g_0,f_{g_0}}(h)=0\right\}.
\end{align*}
Let $\tilde{\nu}$ be the $\nu$-functional restricted to $\mathcal{S}_{g_0,f_{g_0}}$. Obviously, $\tilde{\nu}$ is analytic\index{analytic} since $\nu$ is.
By the first variational formula\index{first variation!of $\nu$} in Lemma \ref{firstnu},
 the $L^2$-gradient of $\nu$ is (up to a constant factor) given by $\nabla\nu(g)=[\tau(\ric_g+\nabla^2 f_g)-\frac{1}{2}g]e^{-f_g}$. 
It vanishes at $g_0$. On the neighbourhood $\mathcal{U}$, we have the uniform estimate
\begin{align}\label{H^2estimatefornu}\left\|\nabla\nu(g_1)-\nabla\nu(g_2)\right\|_{L^2}\leq C\left\|g_1-g_2\right\|_{H^2},
\end{align}
which holds by Taylor expansion\index{Taylor expansion} and Lemma \ref{f',tau}. The $L^2$-gradient of $\tilde{\nu}_-$ is given by the projection\index{projection} of $\nabla\nu$ to $\delta^{-1}_{g_0,f_{g_0}}(0)$. Therefore, (\ref{H^2estimatefornu}) also holds for $\nabla\tilde{\nu}$.
The linearization\index{linearization} of $\tilde{\nu}$ at $g_0$ vanishes on $\R\cdot \ric_{g_0}$ and equals 
$$\frac{\tau_{g_0}}{(4\pi\tau_{g_0})^{n/2}}e^{-f_0}(-\frac{1}{2}\Delta_{f}+\mathring{R})$$
on the $L^2(e^{-f_{g_0}}\dv)$-orthogonal complement\index{L@$L^2$-orthogonal!complement} of $\R\cdot\ric_{g_0}$ in $\delta^{-1}_{g_0,f_{g_0}}(0)$, see Remark \ref{decompremark}. Let us denote this operator by $D$. By ellipticity\index{ellipticity},
\begin{align*}D:(\delta^{-1}_{g_0,f_{g_0}}(0))^{{2,\alpha}}\to(\delta^{-1}_{g_0,f_{g_0}}(0))^{{0,\alpha}}
\end{align*}
is Fredholm. It also satisfies the estimate $\left\|Dh\right\|_{L^2}\leq C \left\|h\right\|_{H^2}$. By a general Lojasiewicz-Simon inequality \cite[Theorem 7.3]{CM12}, there exists a constant $\sigma\in [1/2,1)$ such that the inequality $|\nu(g)-\nu(g_0)|^{\sigma}\leq \left\|\nabla\tilde{\nu}_-(g)\right\|_{L^2}$
holds for any $g\in\mathcal{S}_{g_0,f_{g_0}}$. Since
\begin{align*}\left\|\nabla\tilde{\nu}(g)\right\|_{L^2}\leq \left\|\nabla\nu(g)\right\|_{L^2}\leq  C \left\|\tau(\ric_g+\nabla^2 f_g)-\frac{1}{2}g\right\|_{L^2},
\end{align*}
(\ref{gradientnu2}) holds on all $g\in\mathcal{S}_{g_0,f_{g_0}}$. By diffeomorphism invariance\index{diffeomorphism!invariance}, it holds on all $g\in \mathcal{U}$.
\end{proof}
\section{Dynamical stability and instability}
In this section, we prove the main theorems of the paper. 
We consider the $\tau$-flow
\begin{align}\label{tauricciflow}\dot{g}(t)=-2\ric_{g(t)}+\frac{1}{\tau_{g(t)}}g(t)
\end{align}
which is well-defined in a neighbourhood of a gradient shrinking Ricci soliton. Observe that $\nu$ is nondecreasing under the $\tau$-flow.
We also construct a modified $\tau$-flow as follows: Let $\varphi_t\in \Diff(M)$, $t\geq 1$ be the family of diffeomorphisms generated by
$X(t)=-\gradient_{g(t)}f_{g(t)}$ and define $\tilde{g}(t):=\varphi_t^*g(t)$, where $g(t)$ is a solution of $\eqref{tauricciflow}$. Then we have
\begin{align}\label{modifiedtauflow}\frac{d}{dt}\tilde{g}(t)=-2(\ric_{\tilde{g}(t)}+\nabla^2 f_{\tilde{g}(t)})+\frac{1}{\tau_{\tilde{g}(t)}}\tilde{g}(t).
\end{align}
This is the gradient flow of $\tau$ with respect to the weighted $L^2$-measure.

\begin{lem}\label{shorttimeestimates3}Let $(M,g)$ be a gradient shrinking Ricci soliton and let $k\geq3$. Then for each $C^{k}$-neighbourhood $\mathcal{U}$ of $g$, there exists a $C^{k+5}$-neighbourhood $\mathcal{V}$ of $g$ such that
the modified $\tau$-flow \eqref{modifiedtauflow}, starting at any metric in $\mathcal{V}$, stays in $\mathcal{U}$ for all $t\in[0,1]$.
\end{lem}
\begin{proof}  Let us denote the $\epsilon$-ball with respect to the $C^k_{g}$-norm by $\mathcal{B}_{\epsilon}^k$.
Without loss of generality, we may assume that $\mathcal{U}$
 is      of the form $\mathcal{U}=\mathcal{B}_{\epsilon}^k$ for some $\epsilon>0$.
Throughout the proof, let us assume that we are in a neighbourhood of $g$ such that 
Lemma \ref{analyticnu} and Lemma \ref{f',tau} hold. All covariant derivatives, Laplacians and norms in this proof are taken with respect to $g(t)$ (resp.\ $\tilde{g}(t)$).
 Along the (unmodified) $\tau$-flow, we have the evolution equations
\begin{align*}\partial_t R&=-\Delta R+R*R+\frac{2}{\tau}R,\\
              \partial_t \ric&=-\Delta\ric+R*\ric,\\
              \partial_t \frac{1}{2\tau}g&=-\frac{\partial_t\tau}{2\tau^2}g+\frac{1}{2\tau}\left(-2\ric+\frac{1}{\tau}g\right),\\
              \partial_t \nabla^2 f&=\nabla\ric*\nabla f+\nabla^2\partial_t f,
\end{align*}
where $*$ is Hamilton's notation for a combination of tensor products and contractions. The first two formulas follow from rescaling the evolution equations for the standard Ricci flow \cite[pp.\ 26-28]{Bre10}.
The third formula is clear and the last one follows from the first variation of the Hessian \cite[Lemma A.2]{Kro13a}.
The evolution equation for the Riemann tensor yields the evolution inequality
\begin{align*}\partial_t|\nabla^iR|^2\leq -\Delta|\nabla^iR|^2+\sum_{j=1}^{i-1}C_{ij}|\nabla^jR||\nabla^{i-j}R||\nabla^i R|+C_{i0}\left(|R|+\frac{1}{\tau}\right)|\nabla^i R|^2.
\end{align*}
From the maximum principle, we obtain the following. Suppose we have a $\tau$-flow $g(t)$ defined on $[0,T]$, $T\leq 1$ and the bounds
\begin{align}\label{conditionK}\sup_{p\in M}|R_{g(t)}|_{g(t)}\leq K,\qquad \frac{1}{\tau_{g(t)}}\leq K,\qquad\sup_{p\in M}|\nabla^iR_{g(0)}|_{g(0)}\leq K
\end{align}
for all $t\in[0,T]$ and $i\leq k+3$. Then there exists a constant $\tilde{K}(K,n,k)$ such that
\begin{align}\label{conditiontildeK}\sup_{p\in M}|\nabla^i R_{g(t)}|_{g(t)}\leq \tilde{K}
\end{align}
for all $t\in [0,T]$ and $i\leq k+3$.
Furthermore, we have the evolution inequality
\begin{align*}
\partial_t|\ric+\nabla^2 f-1/2\tau \cdot g|^2\leq -\Delta |\ric+\nabla^2 f-1/2\tau \cdot g|^2+2|\ric+\nabla^2 f-1/2\tau \cdot g||(*)|,
\end{align*}
where $(*)$ is given by
\begin{align*}
 (*)=R*\ric+\frac{\partial_t\tau}{2\tau^2}g-\frac{1}{2\tau}\left(-2\ric+\frac{1}{\tau}g\right)+\nabla\ric*\nabla f+\nabla^2\partial_t f+\Delta \nabla^2 f.
\end{align*}
We obtain a bound
\begin{align*}
 |(*)|\leq C(\left\|f\right\|_{C^4},\left\|R\right\|_{C^1},\left\|\partial_t f\right\|_{C^2}),
\end{align*}
where we used the estimate $ |\partial_t\tau|\leq C\left\|\ric-\frac{1}{2\tau}g\right\|_{L^2}.$
Similarly, we have 
\begin{align*}
\partial_t|\nabla^i(\ric+\nabla^2 f-1/2\tau \cdot g)|^2\leq& -\Delta |\nabla^i(\ric+\nabla^2 f-1/2\tau \cdot g)|^2\\
                                                           &+2|\nabla^i(\ric+\nabla^2 f-1/2\tau \cdot g)||(**)|,
\end{align*}
where we have the bound
\begin{align}\label{doublestar}
 |(**)|\leq C(\left\|f\right\|_{C^{i+4}},\left\|R\right\|_{C^{i+1}},\left\|\partial_t f\right\|_{C^{i+2}}).
\end{align}
It remains to control the norms of $f$ and $\partial_tf$. By the differential equation
\begin{align}\label{PDE}\tau(2\Delta f+|\nabla f|^2-\scal)-f+n+\nu=0
 \end{align}
 and elliptic regularity,
\begin{align}\label{iteration}
\left\|f\right\|_{C^{i,\alpha}}\leq C(\left\|f\right\|_{C^{i-1,\alpha}}+\left\|\scal\right\|_{C^{i-2,\alpha}}+|\nu+n|).
\end{align}
From differentiating \eqref{PDE}, we obtain
\begin{align}
 \left\|\partial_t f\right\|_{C^{i,\alpha}}\leq C(\left\| \ric-1/2\tau \cdot g\right\|_{C^{i,\alpha}},\left\|f\right\|_{C^{i,\alpha}},\left\|R\right\|_{C^{i-2,\alpha}}).
\end{align}
Suppose now again that \eqref{conditionK} (and therfore, also \eqref{conditiontildeK}) holds.
Using an iteration argument in \eqref{iteration} and descending to $C^i$-norms we obtain that    
\begin{align*}
\left\| f\right\|_{C^{i+4}}\leq C(\left\|f\right\|_{C^{2,\alpha}},\left\| R\right\|_{C^{i+3}}),\qquad \left\|\partial_t f\right\|_{C^{i+2}}\leq C(\left\|f\right\|_{C^{2,\alpha}},\left\| R\right\|_{C^{i+3}}).
\end{align*}
For all $i\leq k$, we thus have, by \eqref{conditiontildeK} and \eqref{doublestar},
\begin{equation}\begin{split}\label{hihi}
\partial_t|\nabla^i(\ric+\nabla^2 f-1/2\tau \cdot g)|^2\leq& -\Delta |\nabla^i(\ric+\nabla^2 f-1/2\tau \cdot g)|^2\\
                                                           &+C(\left\|f\right\|_{C^{2,\alpha}},\tilde{K})|\nabla^i(\ric+\nabla^2 f-1/2\tau \cdot g)|
\end{split}
\end{equation}
along the $\tau$ flow. By diffeomorphism invariance, conclusion \eqref{conditiontildeK} and the inequality \eqref{hihi} also hold for the modified $\tau$-flow. Let us denote the modified $\tau$-flow by $\tilde{g}$.
 Choose the $C^{k+5}$-neighbourhood $\mathcal{V}$ so small that we have the bounds \eqref{conditionK} for some constant $K>0$ and any modified $\tau$-flow starting in $\mathcal{V}$ as long as the flow stays in $\mathcal{U}$. Then \eqref{conditiontildeK} holds.
Let $\epsilon_1>0$. By \eqref{hihi} and the maximum principle, we can choose $\delta_1=\delta_1(\epsilon_1,\mathcal{U},\mathcal{V})>0$ so small that if
\begin{align*}
 |\nabla^i(\ric_{\tilde{g}(0)}+\nabla^2 f_{\tilde{g}(0)}-1/2\tau_{\tilde{g}(0)} \cdot \tilde{g}(0))|\leq \delta_1,
\end{align*}
then
\begin{align*}
 |\nabla^i(\ric_{\tilde{g}(t)}+\nabla^2 f_{\tilde{g}(t)}-1/2\tau_{\tilde{g}(t)} \cdot \tilde{g}(t))|\leq \epsilon_1,
\end{align*}
for $i\leq k$ and $t\in[0,1]$ as long as the flow stays in $\mathcal{U}$.    Let $T>0$ be the maximal time such that the modified $\tau$-flow stays in $\mathcal{U}$. Suppose that $T\leq1$.
By integration,
\begin{align*}
\left\|\tilde{g}(T)-g\right\|_{C^{k}_{g}}&\leq \left\|\tilde{g} (0)-g\right\|_{C^{k}_{g}}+\int_0^T\frac{d}{dt}\left\|\tilde{g}(t)-\tilde{g}(0)\right\|_{C^{k}_{g}}dt     \\
                                &       \leq      \delta(\mathcal{V})+C(\mathcal{U})\int_0^T\left\|\dot{\tilde{g}}(t)\right\|_{C^{k}_{\tilde{g}(t)}}dt  \\
                                &       \leq      \delta(\mathcal{V})+C(\mathcal{U})  \cdot k\cdot \epsilon_1\leq\frac{\epsilon}{2},
  \end{align*}
 provided that we have chosen $\epsilon_1$ and $\delta(\mathcal{V})$ small enough. This contradicts the maximality of $T$ and proves the lemma.
\end{proof}
\begin{lem}[Shi estimates for the $\tau$-flow]\label{nablacurvature3}Let $g(t)$, $t\in [0,T]$ be a solution of the $\tau$-flow (\ref{tauricciflow}) and suppose that
\begin{align*}\sup_{p\in M}|R_{g(t)}|_{g(t)}+\frac{1}{\tau_{g(t)}}\leq T^{-1}\qquad \forall t\in[0,T].
\end{align*}
Then for each $k\geq 1$, there exists a constant $C(k)$ such that
\begin{align*} \sup_{p\in M}|\nabla^kR_{g(t)}|_{g(t)}\leq C(k)\cdot T^{-1} t^{-k/2}\qquad\forall t\in(0,T].
\end{align*}
\end{lem}
\begin{proof}
 See \cite[Lemma 6.5.6]{Kro13}.
\end{proof}
Estimates of that type are well-kown for the standard Ricci flow \cite[Theorem 7.1]{Ham95}. Note that we do not make assumptions on the derivatives of the curvature at $t=0$.
\begin{thm}[Dynamical stability]\label{dynamicalstability}\index{stable!dynamically}\index{modulo diffeomorphism}Let $(M,g)$ be a gradient shrinking Ricci soliton and let $k\geq3$. Suppose that $g$ is a local maximizer of $\nu$.
Then for every $C^{k}$-neighbourhood $\mathcal{U}$
of $g$, there exists a $C^{k+5}$-neighbourhood $\mathcal{V}$ such that the following holds:

For any metric $g_0\in\mathcal{V}$, there exists a $1$-parameter family\index{1@$1$-parameter family} of diffeomorphisms $\varphi_t$  such that for the $\tau$-flow $(\ref{tauricciflow})$
 starting at $g_0$, the modified flow $\varphi_t^*g(t)$
stays in $\mathcal{U}$ for all time and converges to a gradient shrinking Ricci soliton $g_{\infty}$ in $\mathcal{U}$ as $t\to\infty$.
The convergence is of polynomial rate, i.e.\ there exist constants $C,\alpha>0$ such that
\begin{align*}\left\|\varphi_t^*g(t)-g_{\infty}\right\|_{C^k}\leq C(t+1)^{-\alpha}.
\end{align*}
\end{thm}
\begin{proof}Without loss of generality, we may assume that $\mathcal{U}=\mathcal{B}_{\epsilon}^{k}$
and that $\epsilon>0$ is so small that Theorem \ref{nonintegrablegradientII} holds on $\mathcal{U}$.

By Lemma \ref{shorttimeestimates3}, we can choose a small neighbourhood $\mathcal{V}$ such that the modified $\tau$-flow \eqref{modifiedtauflow}, starting at any metric $g\in\mathcal{V}$ stays in 
$\mathcal{B}^k_{\epsilon/4}$
up to time $1$. Let $T\geq 1$ be the maximal time such that any solution of \eqref{modifiedtauflow}, starting in $\mathcal{V}$,
stays in $\mathcal{U}$. By definition of $T$, we have uniform bounds
\begin{align*}\sup_{p\in M}|R_{\tilde{g}(t)}|_{\tilde{g}(t)}&\leq C_1\qquad \forall t\in[0,T),\\
              |\tau_{\tilde{g}(t)}|&\leq C_2\qquad \forall t\in[0,T).
\end{align*}
By Lemma \ref{nablacurvature3} and diffeomorphism invariance, we have
\begin{align}\label{nablal_RII}\sup_{p\in M}|\nabla^l R_{\tilde{g}(t)}|_{\tilde{g}(t)}\leq C(l)\qquad \forall t\in[1,T).
\end{align}
Because $f_{\tilde{g}(t)}$ satisfies the equation $\tau(2\Delta f+|\nabla f|^2-\scal)-f+n+\nu=0$, we obtain
\begin{align}\label{nablal_fII}\sup_{p\in M}|\nabla^l f_{\tilde{g}(t)}|_{\tilde{g}(t)}\leq \tilde{C}(l)\qquad \forall t\in[1,T).
\end{align}
 We have
\begin{align*}\left\|\tilde{g}(T)-g\right\|_{C^k}\leq& \left\|\tilde{g}(1)-g\right\|_{C^k}+\int_1^{T}\left\| \dot{\tilde{g}}(t)\right\|_{C^{k}}dt
                                                               \leq\frac{\epsilon}{4}+\int_1^{T}\left\| \dot{\tilde{g}}(t)\right\|_{C^{k}}dt.
\end{align*}
By interpolation (c.f. \cite[Corollary 12.7]{Ham82}), (\ref{nablal_RII}) and (\ref{nablal_fII}), we have 
\begin{align*}\left\| \dot{\tilde{g}}(t)\right\|_{C^k}\leq C_3\left\| \dot{\tilde{g}}(t)\right\|^{1-\eta}_{L^2}
\end{align*}
for $t\in [1,T)$ and for $\eta$ as small as we want. In particular, we can assume that $\theta:=1-\sigma(1+\eta)>0$, where $\sigma$ is the constant appearing in the Lojasiewicz-Simon inequality\index{Lojasiewicz-Simon inequality} \ref{nonintegrablegradientII}.
By the first variation of $\nu$\index{first variation!of $\nu$},
\begin{align*}\frac{d}{dt}\nu(\tilde{g}(t))\geq C_4\left\| \dot{\tilde{g}}(t)\right\|^{1+\eta}_{L^2}\left\| \dot{\tilde{g}}(t)\right\|^{1-\eta}_{L^2}.
\end{align*}
By Theorem \ref{nonintegrablegradientII},
\begin{align*}-\frac{d}{dt}|\nu(\tilde{g}(t))-\nu(g)|^{\theta}&=\theta|\nu(\tilde{g}(t))-\nu(g)|^{\theta-1}\frac{d}{dt}\nu(\tilde{g}(t))\\
                                                                    &\geq C_5|\nu(\tilde{g}(t))-\nu(g)|^{-\sigma(1+\eta)}\left\| \dot{\tilde{g}}(t)\right\|^{1+\eta}_{L^2}\left\| \dot{\tilde{g}}(t)\right\|^{1-\eta}_{L^2}\\
                                                       &\geq C_6\left\| \dot{\tilde{g}}(t)\right\|_{C^k}.
\end{align*}
for $t\in [1,T)$.
Hence by integration,
\begin{align*}
\int_1^{T}\left\| \dot{\tilde{g}}(t)\right\|_{C^k}dt\leq \frac{1}{C_6}|\nu(\tilde{g}(1))-\nu(g)|^{\theta}\leq \frac{1}{C_6}|\nu(\tilde{g}(0))-\nu(g)|^{\theta}\leq\frac{\epsilon}{4},
\end{align*}
provided that $\mathcal{V}$ is small enough.
Thus, $T=\infty$ and $\tilde{g}(t)$ converges to some limit metric $g_{\infty}\in\mathcal{U}$ as $t\to\infty$. By the Lojasiewicz-Simon inequality\index{Lojasiewicz-Simon inequality}, we have
\begin{align*}\frac{d}{dt}|\nu(\tilde{g}(t))-\nu(g)|^{1-2\sigma}\geq C_7,
\end{align*}
which implies
\begin{align*}|\nu(\tilde{g}(t))-\nu(g)|\leq C_8(t+1)^{-\frac{1}{2\sigma-1}}.
\end{align*}
Therefore, $\nu(g_{\infty})=\nu(g)$, so $g_{\infty}$ is a gradient shrinking Ricci soliton, since it is also a local maximum of $\nu$. The convergence is of polynomial rate, since for $t_1<t_2$,
\begin{align*}\left\|\tilde{g}(t_1)-\tilde{g}(t_2)\right\|_{C^k}\leq C_9|\nu(\tilde{g}(t_1))-\nu(g)|^{\theta}\leq C_{10}(t_1+1)^{-\frac{\theta}{2\sigma-1}}.
\end{align*}
The assertion follows from $t_2\to\infty$. By the above arguments, one also sees that $\tilde{g}(t)$ converges in any $C^k$-norm and therefore, the limit metric is smooth.
\end{proof}
\begin{thm}[Dynamical instability]\label{dynamicalinstabilitymodulodiffeo}\index{unstable!dynamically}\index{modulo diffeomorphism}Let $(M,g)$ be a gradient shrinking Ricci soliton that is not a local maximizer of $\nu$. 
 Then there exists a nontrivial ancient $\tau$-flow\index{ancient} $g(t)$,
 $t\in (-\infty,0]$ and a $1$-parameter family\index{1@$1$-parameter family} of diffeomorphisms $\varphi_t$, $t\in (-\infty,0]$ such that $\varphi_t^*g(t)\to g$ as $t\to\infty$.
\end{thm}
\begin{proof}Let $g_i\to g$ in $C^k$ and suppose that $\nu(g_i)>\nu(g)$ for all $i$.
Let $\tilde{g}_i(t)$ be a solution of \eqref{modifiedtauflow} starting at $g_i$.
Then by Lemma \ref{shorttimeestimates3}, $\bar{g}_i=g_i(1)$ converges to $g$ in $C^{k-5}$ and by monotonicity\index{monotonicity}, $\nu(\bar{g}_i)>\nu(g)$ as well.
Let $\epsilon>0$ be so small that Theorem \ref{nonintegrablegradientII} holds on $\mathcal{B}^{k-5}_{2\epsilon}$.
Then we have the differential inequality\index{differential inequality}
\begin{align*}\frac{d}{dt}(\nu(\tilde{g}_i(t))-\nu(g))^{1-2\sigma}\geq -C_1,
\end{align*}
from which we obtain
\begin{align*}[(\nu(\tilde{g}_i(t))-\nu(g))^{1-2\sigma}-C_1(s-t)]^{-\frac{1}{2\sigma-1}} \leq (\nu(\tilde{g}_i(s))-\nu(g)),
\end{align*}
as long as $\tilde{g}_i(t)$ stays in $\mathcal{B}^{k-5}_{2\epsilon}$.
Thus, there exists a $t_i$ such that
\begin{align*}\left\|\tilde{g}_i(t_i)-g\right\|_{C^{k-5}}=\epsilon,
\end{align*}
and $t_i\to \infty$. If $\left\{t_i\right\}$ was bounded, $\tilde{g}_i(t_i)\to g$ in $C^{k-5}$. By interpolation,
\begin{align*}\left\|\dot{\tilde{g}}_i(t)\right\|_{C^{k-5}}\leq C_2 \left\|\dot{\tilde{g}}_i(t)\right\|_{L^2}^{1-\eta}
\end{align*}
for $\eta>0$ as small as we want. We may assume that $\theta=1-\sigma(1+\eta)>0$.
By Theorem \ref{nonintegrablegradientII} , we have the differential inequality\index{differential inequality}
\begin{align*}\frac{d}{dt}(\nu(\tilde{g}_i(t))&-\nu(g))^{\theta}\geq C_3\left\|\dot{\tilde{g}}_i(t)\right\|_{L^2}^{1-\eta},
\end{align*}
if $\nu(\tilde{g}_i(t))>\nu(g)$. Thus,
\begin{align}\label{nontrivial4}\epsilon=\left\|\tilde{g}_i(t_i)-g\right\|_{C^{k-5}}\leq\left\|\bar{g}_i-g\right\|_{C^{k-5}}+C_4(\nu(\tilde{g}_i(t_i))-\nu(g))^{\theta}.
\end{align}
Now put $\tilde{g}_i^s(t):=\tilde{g}_i(t+t_i)$, $t\in[T_i,0]$, where $T_i=1-t_i\to-\infty$. We have
\begin{align*}\left\|\tilde{g}^s_i(t)-g\right\|_{C^{k-5}}&\leq\epsilon\qquad\forall t\in[T_i,0],\\
                \tilde{g}_i^s(T_i)&\to g \text{ in }C^{k-5}.
\end{align*}
Because the embedding $C^{k-6}(M)\subset C^{k-5}(M)$ is compact\index{compact embedding}, we can choose a subsequence\index{subsequence} of the $\tilde{g}_i^s$, converging in $C^{k-6}_{loc}(M\times (-\infty,0])$ to an ancient flow\index{ancient} $\tilde{g}(t)$, $t\in (-\infty,0]$,
which satisfies the differential equation\index{differential equation}
\begin{align*} \dot{\tilde{g}}(t)=-2\left(\ric_{\tilde{g}(t)}-\frac{1}{2\tau_{\tilde{g}(t)}}\tilde{g}(t)+\nabla^2 f_{\tilde{g}(t)}\right).
\end{align*}
From taking the limit $i\to\infty$ in \eqref{nontrivial4}, we have $\epsilon\leq C_4(\nu(\tilde{g}(0))-\nu(g))^{\beta/2}$ which shows that the Ricci flow is nontrivial.
For $T_i\leq t$, the Lojasiewicz-Simon inequality\index{Lojasiewicz-Simon inequality} implies
\begin{align*}\left\|\tilde{g}^s_i(T_i)-\tilde{g}_i^s(t)\right\|_{C^{k-6}}\leq& C_4(\nu(\tilde{g}_i(t+t_i))-\nu(g))^{\theta}\\
                                                           \leq&C_4[-C_1t+ (\nu(\tilde{g}_i(t_i))-\nu(g))^{1-2\sigma}]^{-\frac{\theta}{2\sigma-1}}\\
                                                           \leq&[-C_5t+C_6]^{-\frac{\theta}{2\sigma-1}}.
\end{align*}
Thus,
\begin{align*}\left\|g-\tilde{g}(t)\right\|_{C^{k-6}}\leq &\left\|g-\tilde{g}^s_i(T_i)\right\|_{C^{k-6}}+[-C_5t+C_6]^{-\frac{\theta}{2\sigma-1}}
                                      +\left\|\tilde{g}_i^s(t)-\tilde{g}(t)\right\|_{C^{k-6}}.
\end{align*}
It follows that $\left\|g-\tilde{g}(t)\right\|_{C^{k-6}}\to0$ as $t\to-\infty$. 
\end{proof}
\begin{rem}For any $\tau$-flow $g(t)$ we obtain a solution $\hat{g}(t)$ of the normalized Ricci flow \eqref{volumepreservingricciflow} by projecting to the subset $\mathcal{M}_{c}$ of metrics of volume $c$ and by rescaling the time parameter suitably.
Now if $g(t)$ converges to some Ricci soliton as $t\to \pm\infty$, the same holds for $\hat{g}(t)$. This shows that Theorem \ref{dynamicalstability} and Theorem \ref{dynamicalinstabilitymodulodiffeo} also hold when replacing the $\tau$-flow by the volume-normalized Ricci flow.
In this way, we obtain the theorems as stated in the introduction.
\end{rem}
\begin{rem}
 All known nontrivial Ricci solitons are dynamically unstable (c.f.\ Remark \ref{linstabilityremark}).
\end{rem}

       \section{The integrability condition}
Let $g_t$ a $C^1$-curve of Ricci solitons through $g$ and suppose that all $g_t$ are of the same volume $c$. By projecting to the affine slice, we obtain a $C^1$-curve of Ricci solitons
\begin{align*}
 \tilde{g}_t\subset \mathcal{M}_c\cap \left\{g+k\mid \delta_{g,f_{g}}(k)=0\right\}
\end{align*}
where $\mathcal{M}_c$ is the space of metrics of volume $c$.
 We have \begin{align*}
          \frac{d}{dt}\bigg\vert_{t=0}\tilde{g}_t=h\in  V=\left\{h\in \Gamma(S^2M)\text{ }\bigg\vert \text{ }\delta_f(h)=0 \text{ and }\int_M \langle\ric,h\rangle e^{-f}\dv=0\right\}
\end{align*}
 and $\nu''(h)=0$ because $\nu$ is constant along $\tilde{g}_t$. This motivates the following definition:
\begin{defn}
 Let $(M,g)$ be a gradient shrinking Ricci soliton and let $N$ be the stability operator of Proposition \ref{secondnu}. We call $h\in C^{\infty}(S^2M)$ an infinitesimal solitonic deformation if $h\in V$ and $N(h)=0$. An infinitesimal solitonic deformation is called integrable if there exists a curve of Ricci solitons $g_t$ through $g=g_0$
such that $\frac{d}{dt}|_{t=0}g_t=h$.
\end{defn}
If $g$ is Einstein, this generalizes the notion of infinitesimal Einstein deformations (IED). 
Recall that an IED is a trace-free and divergence-free (TT) tensor which lies in the kernel of the Einstein operator $\Delta_E =\Delta-2\mathring{R}$ \cite[p.\ 347]{Bes08}.
An IED is called integrable if there exists a curve of Einstein metrics tangent to it.
\begin{lem}\label{ISDIED}
 Let $(M,g)$ be an Einstein manifold with Einstein constant $\mu$. Let $IED$ be the space of infinitesimal Einstein deformations and $ISD$ the space of infinitesimal solitonic deformations. Then we have
\begin{align*}
 ISD=IED\oplus \left\{\mu v\cdot g+\nabla^2 v|v\in C^{\infty}(M),\Delta v=2\mu v\right\}.
\end{align*}
\end{lem}
\begin{proof}Since $g$ is Einstein, we have $\Delta_E=-\frac{1}{2}N$ on $V$, where $V$ is as above. The space $V$ splits 
as
\begin{align*}
 V=TT\oplus \left\{F(v)=(\Delta v-\mu v)g+\nabla^2 v\text{ }\bigg\vert\text{ } \int_M v\dv=0\right\}
\end{align*}
The kernel of $\Delta_E$ on TT-tensors is $IED$ by definition. On the second component, we have $\Delta_E (F(v))=F((\Delta-2\mu)v)$ \cite[p.\ 7]{CHI04}.
See also \cite[p.\ 106]{Kro13} for details.
\end{proof}
IED's and their integrablity were studied by Koiso \cite{Koi78,Koi80,Koi82,Koi83}, see also \cite[Chapter 12]{Bes08}.
Recently, Podest\`a and Spiro \cite{PS13} generalized some of Koiso's results to the case of Ricci solitons. For a given Ricci soliton $g$, they constructed a slice $\mathcal{S}_{g,f_{g}}$ in the space of metrics tangent to $\delta^{-1}_{g,f_{g}}(0)$ such that
the follwing holds: For any $s\geq [\frac{n}{2}]+3$, there exists a $H^{s}$-neighbourhood $\mathcal{U}\subset \mathcal{S}_{g,f_{g}}$ of $g$ and a finite-dimensional submanifold $\mathcal{Z}\subset\mathcal{U}$ such that
$T_{g}\mathcal{P}=\kernel(N|_{V})\oplus\R\cdot\ric_g$ and the set of Ricci solitons $\mathcal{P}\subset\mathcal{Z}$ is a real analytic subset of $\mathcal{Z}$ \cite[Theorem 3.4]{PS13}. Any metric which is
$H^s$-close to $g$ is isometric to a unique metric in $S_{g,f_g}$.

For the main theorem in this section, we impose the technical condition that all infinitesimal solitonic deformations are integrable. If this condition holds, $\mathcal{P}=\mathcal{Z}$ provided that $\mathcal{U}$ is small enough.
In particular, the set of Ricci solitons in $\mathcal{U}$ is a finite-dimensional manifold.

The slice used in \cite{PS13} is constructed via the exponential map of the weak
Riemannian structure on $\mathcal{M}$. Hence it
 differs from the affine slice $\tilde{\mathcal{S}}_{g,f_g}$ that we use in this paper. 
However, we can identify the Ricci solitons in $\mathcal{S}_{g,f_g}$ and $\tilde{\mathcal{S}}_{g,f_g}$ via the map $\Psi:\mathcal{P}\to\tilde{\mathcal{S}}_{g,f_g}$
which associates to $g_1\in \mathcal{S}_{g,f_g}$ the unique metric $\tilde{g}_1\in \tilde{\mathcal{S}}_{g,f_g}$ isometric to $g_1$. Since $\mathcal{P}$ is finite-dimensional,
it also consists of all Ricci solitons in a suitable $C^{2,\alpha}$-neighbourhood of $g$.
By \cite[Lemma A.5]{Ach12}, $\Psi$ is a diffeomorphism onto its image and the set $\tilde{\mathcal{P}}=\Psi(\mathcal{P})$, consisting of all Ricci solitons in $\tilde{\mathcal{S}}_{g,f_g}$, is a finite-dimensional manifold.
\begin{lem}\label{f'',tau}Let $(M,g_0)$ be a gradient shrinking Ricci soliton. Then there exists a $C^{2,\alpha}$-neighbourhood $\mathcal{U}$ of $g_0$ and a constant $C>0$ such that
\begin{align*}\left\|\frac{d^2}{dtds}\bigg\vert_{t,s=0}f_{g+sk+th}\right\|_{H^1}\leq C \left\|k\right\|_{C^{2,\alpha}}\left\|h\right\|_{H^1},\qquad \left|\frac{d^2}{dtds}\bigg\vert_{t,s=0}\tau_{g+sk+th}\right|\leq C\left\|k\right\|_{C^{2,\alpha}}\left\|h\right\|_{H^1} .
\end{align*}
for all $g\in \mathcal{U}$.
\end{lem}
\begin{proof}
 Let the maps $P,L$ and the space $R$ be as in the proof of the Lemma \ref{analyticnu}. We then have
\begin{align*}
 \left(\frac{d^2}{dtds}\bigg\vert_{t,s=0}f_{g+sk+th},\frac{d^2}{dtds}\bigg\vert_{t,s=0}\tau_{g+sk+th}\right)=\frac{d}{ds}\bigg\vert_{s=0}dP_{g+sk}(h).
\end{align*}
Let us denote $A_g=dL_{g,f_{g},\tau_{g}}|_R$ and $B_g=dL_{g,f_{g},\tau_{g}}|_{C^{2,\alpha}(S^2M)}$, so that we can rewrite \eqref{dP} as
\begin{align}\label{dP2}dP_{g}=-A_g^{-1}\circ B_g.
\end{align}
By differentiating,
\begin{align*}
 \frac{d}{ds}\bigg\vert_{s=0}dP_{g+sk}=-A_g^{-1}\circ \left(\frac{d}{ds}\bigg\vert_{t=0}A_{g+sk}\right)\circ A_g^{-1}\circ B_g-A_g^{-1}\circ \left(\frac{d}{ds}\bigg\vert_{s=0}B_{g+sk}\right).
\end{align*}
The derivatives of $A_g$ and $B_g$ contain the second variation of the Laplacian and the scalar curvature, which can be schematically written as
\begin{align*}\frac{d^2}{dsdt}\bigg\vert_{s,t=0}\Delta_{g+sk+th}f&=\nabla k*h*\nabla f+k*\nabla h*\nabla f,\\
\frac{d^2}{dsdt}\bigg\vert_{s,t=0}\scal_{g+sk+th}&=\nabla^2 k*h+k*\nabla^2 h+\nabla k*\nabla h+ R*k*h,
\end{align*} 
see \cite[Lemma A.3]{Kro13a}.
Thus by straightforward calculation and standard estimates,
\begin{align*}
 \left\|\frac{d}{ds}\bigg\vert_{t=0}A_{g+sk}(f)\right\|_{H^{-1}} \leq C\left\| k\right\|_{C^{2,\alpha}}\left\|f\right\|_{H^{1}},\qquad
 \left\|\frac{d}{ds}\bigg\vert_{t=0}B_{g+sk}(h)\right\|_{H^{-1}} \leq C\left\| k\right\|_{C^{2,\alpha}}\left\|h\right\|_{H^{1}} ,
\end{align*}
where $f\in R$ and $h\in C^{2,\alpha}(S^2M)$. By combining these estimates with those for $A_g$ and $B_g$ in the proof of Lemma \ref{f',tau}, we obtain the result.
\end{proof}

\begin{prop}[Estimates of the third variation of $\nu$]\label{thirdnu}\index{third variation!of $\nu$}Let $(M,g_0)$ be a gradient shrinking Ricci soliton. There exists a $C^{2,\alpha}$-neighbourhood $\mathcal{U}$ of $g_0$ such that
$$\left|\frac{d^3}{dt^3}\bigg\vert_{t=0}\nu(g+th)\right|\leq C \left\| h\right\|^2_{H^1}\left\|h\right\|_{C^{2,\alpha}}$$
for all $g\in\mathcal{U}$ and some constant $C>0$.
\end{prop}
\begin{proof}We put $v=\frac{e^{-f}}{(4\pi\tau)^{n/2}}$ and $\nabla \nu=\tau(\ric+\nabla^2f)-\frac{1}{2}g$. Then
\begin{align*}\frac{d^3}{dt^3}\bigg\vert_{t=0}\nu(g+th)=&-\frac{d^2}{dt^2}\bigg\vert_{t=0}\int_M \langle \nabla \nu,h\rangle v \dv\\
                                                =&-\int_M \langle (\nabla \nu)'',h\rangle v \dv-6\int_M \langle \nabla \nu,h\circ h\circ h\rangle v \dv-\int_M \langle \nabla \nu,h\rangle (v \dv)''\\
                                                &+2\int_M \langle (\nabla \nu)',h\circ h\rangle v \dv+2\int_M \langle \nabla \nu,h\circ h\rangle (v \dv)'-\int_M \langle (\nabla \nu)',h\rangle (v \dv)'.
\end{align*}
 Here, $\circ$ denotes the composition of symmetric $(0,2)$-tensors, considered as endomorphisms on $TM$.
We have
\begin{align*}
\frac{d}{dt}\bigg\vert_{t=0}\nabla^2_{g+th}f&=\nabla h*\nabla f,\\
\frac{d}{dt}\bigg\vert_{t=0}\ric_{g+th}&=(\nabla^2*h)+(R*h),
\end{align*}
see \cite[Lemma A.1 and Lemma A.2]{Kro13a}.
The second variation of the Ricci tensor and the Hessian appearing in $(\nabla \nu)''$ are of the form
\begin{align*}
 \frac{d}{ds}\frac{d}{dt}\bigg\vert_{s,t=0}\nabla_{g+sk+th}^2f=&k*\nabla h*\nabla f+\nabla k*h*\nabla f,\\
\frac{d}{ds}\frac{d}{dt}\bigg\vert_{s,t=0}\ric_{g+sk+th}=&k*\nabla^2h+\nabla^2k*h+\nabla k*\nabla h+R*k*h,
\end{align*}
see \cite[Lemma A.3]{Kro13a}.
Now, standard estimates and the Lemmas \ref{f',tau} and \ref{f'',tau} yield an upper bound of the form $C \left\| h\right\|^2_{H^1}\left\|h\right\|_{C^{2,\alpha}}$ for each of these terms.
\end{proof}
\begin{thm}\label{maximalitynu}Let $(M,g_{0})$ be a shrinking gradient Ricci soliton and suppose that all infinitesimal solitonic deformations are integrable.
If $g_{0}$ is linearly stable, then there exists a small $C^{2,\alpha}$-neighbourhood $\mathcal{U}\subset\mathcal{M}$ of $g_{0}$ such that $\nu(g)\leq \nu(g_{0})$ for all $g\in \mathcal{U}$.
\end{thm}
\begin{rem}
 Observe that the converse implication always holds, even if we drop the integrability assumption.
\end{rem}

\begin{proof}
 By diffeomorphism invariance of $\nu$, it suffices to prove the theorem on a slice in the space of metrics. Let
\begin{align*}\mathcal{S}_{g_{0},f_{g_{0}}}=\mathcal{U}\cap\left\{g_{0}+h \text{ }\bigg\vert\text{ } h\in \delta^{-1}_{g_{0},f_{g_{0}}}(0)\right\}
\end{align*}
and let
\begin{align*}
 \mathcal{P}=\left\{g\in \mathcal{S}_{g_{0},f_{g_{0}}}\text{ }\bigg\vert\text{ }\ric_g+\nabla^2 f_g=\frac{1}{2\tau_g}\cdot g\right\}
\end{align*}
be the set of gradient shrinking Ricci solitons in the slice. By integrability, $\mathcal{P}$ is a finite-dimensional manifold with tangent space
\begin{align*}
T_{g_{0}}\mathcal{P}= \R\cdot\ric\oplus\kernel(N|_{\delta^{-1}_{f_{g_{0}}}(0)})
\end{align*}
where $N$ is the stability operator in Proposition \ref{secondnu}. Let $W$ be the $L^2(e^{-f_{g_{0}}}\dv)$-orthogonal complement of $T_{g_{0}}\mathcal{P}$ in $\delta^{-1}_{f_{g_{0}}}(0)$.
Let us pass to $C^{2,\alpha}$-spaces for the rest of the proof.
By the inverse function theorem for Banach manifolds, any ($C^{2,\alpha}$-)metric $g$ in $\mathcal{S}_{g_{0},f_{g_{0}}}$ can be written as $g=\bar{g}+h$, where $\bar{g}\in\mathcal{P}$ and $h\in W$.
By Taylor expansion,
 \begin{align*}\nu(\bar{g}+h)&=\nu(\bar{g})+\frac{1}{2}\frac{d^2}{dt^2}\bigg\vert_{t=0}\nu(\bar{g}+th)+R(\bar{g},h),\\
                           R(\bar{g},h)&=\int_0^1\left(\frac{1}{2}-t+\frac{1}{2}t^2\right)\frac{d^3}{dt^3}\nu(\bar{g}+th)dt.
 \end{align*}
By assumption, $\nu''_{g_{0}}$ is negative definite on $V$ and therefore,
\begin{align}\label{uniformsecondnu}
 \nu''_{g_{0}}(h)=\frac{\tau}{(4\pi\tau)^{n/2}}\left[-\epsilon\int_M|\nabla h|e^{-f}\dv
 +\int_M\langle[(\epsilon-\frac{1}{2})\Delta_f+\mathring{R}]h,h\rangle e^{-f}\dv\right]   \leq -C_1\left\| h\right\|_{H^1}^2
\end{align}
for all $h\in V$. By Taylor expansion again,
 \begin{align}\label{tayloragain}
  |\nu''_{\bar{g}}(h)-\nu''_{g_{0}}(h)|\leq C_2\left\| \bar{g}-g_{0}\right\|_{C^{2,\alpha}}\left\|h\right\|_{H^1}.
 \end{align}
More precisely,
\begin{align*}
 \frac{d}{ds}\bigg\vert_{s=0}N_{g+sk}h=&\nabla^2 k*h+\nabla k*\nabla h+k*\nabla^2 h+R*h*k+\nabla k*\nabla v_{h}+\nabla^2 v'\\
                                       &+\nabla f*\nabla k*h+\nabla f*k*\nabla h+\nabla^2 f'*h+\nabla f'*h
\end{align*}
by straightforward calculation.
By elliptic regularity,
\begin{align*}
 \left\|v_h\right\|_{H^1}\leq C_3\left\|h\right\|_{H^1}
\end{align*}
and by differentiating \eqref{v_h} and Lemma \ref{f',tau},
\begin{align*}
 \left\|v'\right\|_{H^1}\leq C_4\left\|h\right\|_{H^1}\left\|k\right\|_{C^{2,\alpha}}.
\end{align*}
Together with Lemma \ref{f',tau} again, this proves \eqref{tayloragain}. As a consequence,
\begin{align}\label{uniformsecondnu2}
 \nu''_{\bar{g}}(h)\leq -C_1\left\| h\right\|_{H^1}^2
\end{align}
for all $h\in V$ and for all $\bar{g}$ in a sufficiently small neighbourhood of $g_{0}$.
Thus by Proposition \ref{thirdnu},
 \begin{align*}
  \nu(\bar{g}+h)\leq \nu(\bar{g})-(C_5-C_6\left\|h\right\|_{C^{2,\alpha}})\left\|h\right\|_{H^1}^2
 \end{align*}
which shows that any $\nu(g)\leq\nu(g_{0})$ for any $g\in \mathcal{S}_{g_{0},f_{g_{0}}}$ close enough to $g_{0}$ in $C^{2,\alpha}$. By \cite[Lemma A.5]{Ach12}, any metric in $\mathcal{U}$ is isometric
to some metric in $\mathcal{S}_{g_{0},f_{g_{0}}}$, so that this inequality holds for all $g\in \mathcal{U}$.
In particular, it holds for smooth metrics.
\end{proof}
\begin{cor}
Let $(M,g_0)$ be a gradient shrinking Ricci soliton and suppose that all infinitesimal solitonic deformations are integrable. Then, linear stability and dynamical stability are equivalent.
\end{cor}
\begin{rem}
 It would be interesting to find curvature conditions which ensure linear stability of a given Ricci soliton. This would generalize previous results obtained in the Einstein case,
  see \cite{Koi78,Koi83,IN05,DWW05,DWW07,Kro13b}.
 A detailed study of the stability operator $N$ will be the content of our future investigations.
\end{rem}
\begin{proof}[Proof of Theorem \ref{nonintegrableISD}.]
 The Fubini-Study metric is Einstein and its Einstein constant is given by $1/2\tau$. By \cite[pp.\ 25-26]{Kro13a},
there exists an eigenfunction $v\in C^{\infty}(M)$ of the Laplacian with eigenvalue $1/\tau$ such that $\int_M v^3\dv\neq 0$. By Lemma \ref{ISDIED}, $h=v\cdot g_{fs}+2\tau\nabla^2 v$ is an infinitesimal solitonic deformation.
Suppose there exists a curve of Ricci solitons $g_t$ such that $\frac{d}{dt}|_{t=0}g_t=h$. 
 Since $g_{fs}$ is a critical point of $\nu$ and $h$ lies in the kernel of its linearization, $\frac{d}{dt}|_{t=0}\nu(g_t)=\frac{d^2}{dt^2}|_{t=0}\nu(g_t)=0$. The third variation equals
\begin{align*}
 \frac{d^3}{dt^3}\bigg\vert_{t=0}\nu(g_t)=\nu'''(h)=\nu'''(v\cdot g_{fs})=\frac{2n-2}{\volume(M,g_{fs})}\int_M v^3\dv\neq0.
\end{align*}
Here we used the diffeomorphism invariance of $\nu$ and the third variational formula in \cite[Proposition 9.1]{Kro13a}.
This contradicts the fact that $\nu$ must be constant along $g_t$ and thus, $h$ is not integrable.
\end{proof}
\begin{rem}
 It seems likely that $(\CP^n,g_{fs})$ is isolated in the moduli space of Ricci solitons, because all infinitesimal solitonic deformations arise from conformal deformations. 
\end{rem}
\begin{rem}The space $(\CP^n,g_{fs})$ is neutrally linearly stable but dynamically unstable \cite[Corollary 1.11]{Kro13a}.
 \end{rem}
\vspace{3mm}

\textbf{Acknowledgement.} The author would like to thank Christian B\"ar, Stuart Hall and Thomas Murphy for their support and helpful discussions.
Moreover, the author thanks
Sonderforschungsbereich 647
funded by
Deutsche Forschungsgemeinschaft
for financial support.
\newcommand{\etalchar}[1]{$^{#1}$}


\begin{thebibliography}{CCG{\etalchar{+}}07}
 
\providecommand{\url}[1]{\texttt{#1}}
\expandafter\ifx\csname urlstyle\endcsname\relax
  \providecommand{\doi}[1]{doi: #1}\else
  \providecommand{\doi}{doi: \begingroup \urlstyle{rm}\Url}\fi

\bibitem[Ach12]{Ach12}
\textsc{Ache}, Antonio G.:
\newblock On the uniqueness of asymptotic limits of the Ricci flow.
\newblock arXiv preprint arXiv:1211.3387,
\newblock (2012).

\bibitem[Bes08]{Bes08}
\textsc{Besse}, Arthur~L.:
\newblock \emph{{Einstein manifolds. Reprint of the 1987 edition.}}
\newblock {Berlin: Springer}, 2008

\bibitem[Bre10]{Bre10}
\textsc{Brendle}, Simon:
\newblock \emph{{Ricci flow and the sphere theorem.}}
\newblock Graduate Studies in Mathematics 111. Providence, RI: American
  Mathematical Society (AMS). vii, 176 p., 2010

\bibitem[Cao96]{Cao96} 
\textsc{Cao}, Huai-Dong:
\newblock \emph{Existence of gradient K\"ahler-Ricci solitons,}
\newblock Elliptic and Parabolic Methods in Geometry (Minneapolis, MN, 1994) 1--16

\bibitem[Cao10]{Cao10}
\textsc{Cao}, Huai-Dong:
\newblock "Recent progress on Ricci solitons. Recent advances in geometric analysis, 1–38, Adv. Lect. Math. (ALM), 11."
\newblock \emph{Int. Press, Somerville, MA}
\newblock (2010)

\bibitem[CHI04]{CHI04}
\textsc{Cao}, Huai-Song ; \textsc{Hamilton}, Richard  ; \textsc{Ilmanen}, Tom:
\newblock Gaussian densities and stability for some {R}icci solitons.
\newblock arXiv preprint math/0404165
\newblock   (2004).

\bibitem[CH13]{CH13}
\textsc{Cao}, Huai-Dong ; \textsc{He}, Chenxu:
\newblock Linear {S}tability of {P}erelmans $\nu$-entropy on {S}ymmetric spaces
  of compact type.
\newblock arXiv preprint arXiv:1304.2697,
\newblock   (2013).

\bibitem[CZ12]{CZ12}
\textsc{Cao}, Huai-Dong ; \textsc{Zhu}, Meng:
\newblock {On second variation of Perelman's Ricci shrinker entropy.}
\newblock {In: }\emph{Math.\ Ann.\ } \textbf{353} (2012), no. 3, 747--763

\bibitem[CG96]{CG96}
\textsc{Chave}, Thierry ; \textsc{Galliano}, Valent:
\newblock {On a class of compact and non-compact quasi-Einstein metrics and their renormalizability properties.}
\newblock {In: }\emph{Nucl.\ Phys.\ B} \textbf{478} (1996), no. 3, 758--778

\bibitem[CCG{\etalchar{+}}07]{CC07}
\textsc{Chow}, Bennett ; \textsc{Chu}, Sun-Chin ; \textsc{Glickenstein}, David
  ; \textsc{Guenther}, Christine ; \textsc{Isenberg}, James ; \textsc{Ivey},
  Tom ; \textsc{Knopf}, Dan ; \textsc{Lu}, Peng ; \textsc{Luo}, Feng  ;
  \textsc{Ni}, Lei:
\newblock \emph{{The Ricci flow: techniques and applications. Part I: Geometric
  aspects.}}
\newblock Mathematical Surveys and Monographs 135. Providence, RI: American
  Mathematical Society (AMS). xxiii, 536 p., 2007

\bibitem[CM12]{CM12}
\textsc{Colding}, Tobias~H. ; \textsc{Minicozzi}, William~P.:
\newblock {On uniqueness of tangent cones for Einstein manifolds}.
\newblock {In: }\emph{Invent.\ Math.\ } (2013), 1--74

\bibitem[DWW05]{DWW05}
\textsc{Dai}, Xianzhe ; \textsc{Wang}, Xiaodong  ; \textsc{Wei}, Guofang:
\newblock {On the stability of Riemannian manifold with parallel spinors.}
\newblock {In: }\emph{Invent.\ Math.\ } \textbf{161} (2005), no. 1, 151--176

\bibitem[DWW07]{DWW07}
\textsc{Dai}, Xianzhe ; \textsc{Wang}, Xiaodong  ; \textsc{Wei}, Guofang:
\newblock {On the variational stability of K\"ahler-Einstein metrics.}
\newblock {In: }\emph{Commun.\ Anal.\ Geom.\ } \textbf{15} (2007), no. 4, 669--693

\bibitem[DW11]{DW11}
\textsc{Dancer}, Andrew S. ; \textsc{Wang}, McKenzie Y. :
\newblock {On Ricci solitons of cohomogeneity one.}
\newblock {In: }\emph{Ann.\ Glob.\ Anal.\ Geom.\ } \textbf{39} (2011), no. 3, 259--292

\bibitem[Ebi70]{Eb70}
\textsc{Ebin}, David~G.:
\newblock {The manifold of Riemannian metrics.}
\newblock {In: }\emph{Proc. {S}ymp. {AMS}} Bd.~15, 1970, 11--40

\bibitem[ELM08]{ELM08}
\textsc{Eminenti}, Manolo ; \textsc{La Nave}, Gabriele ; \text{Mantegazza}, Carlo:
\newblock{Ricci solitons: the equation point of view}.
\newblock{In: }\emph{Manuscripta Math.\ } \textbf{127} (2008), no. 3, 345--367

\bibitem[Fri80]{Fri80}
\textsc{Friedan}, Daniel:
\newblock{Nonlinear models in 2 + $\epsilon$ dimensions.}
\newblock {In: }\emph{Phys.\ Rev.\ Lett.\ } \textbf{45} (1980), no. 13, 1057

\bibitem[FS13]{FS13}
\textsc{Futaki}, Akito ; \textsc{Sano}, Yuji:
\newblock{Lower diameter bounds for compact shrinking Ricci solitons}.
\newblock {In: }\emph{Asian Journ.\ Math.\ } \textbf{17} (2013), no. 1, 17--32

\bibitem[GIK02]{GIK02}
\textsc{Guenther}, Christine ; \textsc{Isenberg}, James  ; \textsc{Knopf}, Dan:
\newblock {Stability of the Ricci flow at Ricci-flat metrics.}
\newblock {In: }\emph{Commun.\ Anal.\ Geom.\ } \textbf{10} (2002), no. 4, 741--777

\bibitem[Ham82]{Ham82}
\textsc{Hamilton}, Richard~S.:
\newblock {Three-manifolds with positive Ricci curvature.}
\newblock {In: }\emph{J.\ Differ.\ Geom.\ } \textbf{17} (1982), 255--306

\bibitem[Ham88]{Ham88}
\textsc{Hamilton}, Richard~S.:
\newblock {The Ricci flow on surfaces.}
\newblock {In: }\emph{Contemp.\ Math.\ } \textbf{71} (1988), no. 1

\bibitem[Ham95]{Ham95}
\textsc{Hamilton}, Richard~S.:
\newblock The formation of singularities in the Ricci flow.
\newblock \emph{Surveys in differential geometry} \textbf{2} (1995), 7--136

\bibitem[Has12]{Has12}
\textsc{Haslhofer}, Robert:
\newblock {Perelman's lambda-functional and the stability of Ricci-flat
  metrics.}
\newblock {In: }\emph{Calc.\ Var.\ Partial Differ.\ Equ.\ } \textbf{45} (2012), no. 3-4, 481--504

\bibitem[HHM14]{HHM14}
\textsc{Hall}, Stuart ; \textsc{Haslhofer}, Robert ; \textsc{Murphy}, Thomas:
\newblock {The stability inequality for Ricci-flat cones.}
\newblock {In: }\emph{J.\ Geom.\ Ana.\ } \textbf{24} (2014) no. 1, 472--494

\bibitem[HM11]{HM11}
\textsc{Hall}, Stuart ; \textsc{Murphy}, Thomas:
\newblock {On the linear stability of Kähler-Ricci solitons.}
\newblock Proceedings of the American Mathematical Society \textbf{139} (2011), no. 9, 3327--3337.

\bibitem[HM13]{HM13}
\textsc{Haslhofer}, Robert ; \textsc{M\"uller}, Reto:
\newblock Dynamical stability and instability of {R}icci-flat metrics.
\newblock arXiv preprint arXiv:1301.3219,
\newblock   (2013).

\bibitem[IN05]{IN05}
\textsc{Itoh}, Mitsuhiro ; \textsc{Nakagawa}, Tomomi:
\newblock {Variational stability and local rigidity of Einstein metrics.}
\newblock {In: }\emph{Yokohama Math.\ J.\ } \textbf{51} (2005), no. 2, 103--115

 \bibitem[Ive93]{Ive93}
\textsc{Ivey}, Thomas:
\newblock {Ricci solitons on compact three-manifolds.}
\newblock {In: }\emph{Diff.\ Geom.\ Appl.\ } \textbf{3} (1993), no. 4, 301--307

\bibitem[Kho08]{Kho08}
\textsc{Kholodenko}, Arkady L.
\newblock {Towards physically motivated proofs of the Poincaré and geometrization conjectures.}
\newblock {In: }\emph{J.\ Geom.\ Phys.\ } \textbf{58} (2008), no. 2, 259--290

\bibitem[Koi78]{Koi78}
\textsc{Koiso}, Norihito:
\newblock {Non-deformability of Einstein metrics.}
\newblock {In: }\emph{Osaka J.\ Math.\ } \textbf{15} (1978), 419--433

\bibitem[Koi80]{Koi80}
\textsc{Koiso}, Norihito:
\newblock {Rigidity and stability of Einstein metrics - The case of compact
  symmetric spaces.}
\newblock {In: }\emph{Osaka J.\ Math.\ } \textbf{17} (1980), 51--73

\bibitem[Koi82]{Koi82}
\textsc{Koiso}, Norihito:
\newblock {Rigidity and infinitesimal deformability of Einstein metrics.}
\newblock {In: }\emph{Osaka J.\ Math.\ } \textbf{19} (1982), 643--668

\bibitem[Koi83]{Koi83}
\textsc{Koiso}, Norihito:
\newblock {Einstein metrics and complex structures.}
\newblock {In: }\emph{Invent.\ Math.\ } \textbf{73} (1983), 71--106

\bibitem[Koi90]{Koi90}
\textsc{Koiso}, Norihito:
\newblock On rotationally symmmetric Hamilton’s equation for K\"ahler-Einstein metrics.
\newblock Recent Topics in Diff. Anal. Geom., Adv. Studies Pure Math., \textbf{18-I}, Academic Press, Boston,
MA (1990), 327--337

\bibitem[Kr\"o13a]{Kro13b}
\textsc{Kr\"oncke}, Klaus:
\newblock On the {s}tability of {E}instein {m}anifolds.
\newblock arXiv preprint arXiv:1311.6749,
\newblock   (2013).

\bibitem[Kr\"o13b]{Kro13a}
\textsc{Kr\"oncke}, Klaus:
\newblock {Einstein metrics, Ricci flow and the Yamabe invariant.}
\newblock arXiv preprint arXiv:1312.2224,
\newblock   (2013).

\bibitem[Kr\"o13c]{Kro13}
\textsc{Kr\"oncke}, Klaus:
\newblock \emph{Stability of {E}instein {M}anifolds}, Universit\"at Potsdam,
  PhD thesis, 2013

\bibitem[Per02]{Per02}
\textsc{Perelman}, Grisha:
\newblock {The entropy formula for the Ricci flow and its geometric
  applications.}
\newblock arXiv preprint math/0211159,
\newblock   (2002).

\bibitem[PS10]{PS10}
\textsc{Podest{\`a}}, Fabio ; \textsc{Spiro}, Andrea:
\newblock {K\"ahler-Ricci solitons on homogeneous toric bundles.}
\newblock {In: }\emph{J.\ reine Angew.\ Math.\ } \textbf{642} (2010), 109--127 

\bibitem[PS13]{PS13}
 \textsc{Podest{\`a}}, Fabio ; \textsc{Spiro}, Andrea:
 \newblock {On moduli spaces of Ricci solitons.}
 \newblock {In: }\emph{J.\ Geom.\ Anal.\ } (2013), 1--18

\bibitem[Ses06]{Ses06}
\textsc{Sesum}, Natasa:
\newblock {Linear and dynamical stability of Ricci-flat metrics.}
\newblock {In: }\emph{Duke Math.\ J.\ } \textbf{133} (2006), no. 1, 1--26

\bibitem[SW13]{SW13}
\textsc{Sun}, Song ; \textsc{Wang}, Yuanqi:
\newblock On the {K}\"ahler-{R}icci flow near a {K}\"ahler-{E}instein metric.
\newblock {In: }\emph{J.\ reine Angew.\ Math.\ } (2013)

\bibitem[WZ04]{WZ04}
\textsc{Wang}, Xu-Jia ; \textsc{Zhu}, Xiaohua
\newblock {K\"ahler–Ricci solitons on toric manifolds with positive first Chern class.}
\newblock{In: }\emph{Adv.\ Math.\ } \textbf{188} (2004), no. 1, 87--103

\bibitem[Ye93]{Ye93}
\textsc{Ye}, Rugang:
\newblock {Ricci flow, Einstein metrics and space forms.}
\newblock {In: }\emph{Trans.\ Am.\ Math.\ Soc.\ } \textbf{338} (1993), no. 2, 871--896

\end{thebibliography}
\end{document}